\documentclass[a4paper,reqno,11pt]{amsart}

\usepackage{amsmath,amsfonts,amssymb}
\usepackage{graphics,color}
\usepackage[latin1]{inputenc}
\usepackage[colorlinks = true,
            linkcolor = blue,
            urlcolor  = blue,
            citecolor = blue,
            anchorcolor = blue]{hyperref}
\usepackage{mathrsfs}
\newtheorem{theorem}{Theorem}[section]
\newtheorem{lemma}[theorem]{Lemma}

\makeatletter
\@namedef{subjclassname@2020}{\textup{2020} Mathematics Subject Classification}
\makeatother

\numberwithin{equation}{section}
\newtheorem{proposition}[theorem]{Proposition}

\newtheorem{remark}[theorem]{Remark}
\newcommand{\R}{\mathbb{R}}
\newcommand{\Mcal}{\mathcal{M}}
\newcommand{\Ncal}{\mathcal{N}}
\newcommand{\Fcal}{\mathcal{F}}
\newcommand{\Pcal}{\mathcal{P}}
\newcommand{\Zcal}{\mathcal{Z}}

\newcommand{\Qcal}{\mathcal{Q}}
\newcommand{\dd}{\mathrm{d}}
\newcommand{\eps}{\varepsilon}
\newcommand{\pa}{\partial}
\usepackage{cite}
\numberwithin{equation}{section}

\allowdisplaybreaks[4]
\begin{document}
	
\title[Attractive Bose--Einstein condensates in displaced harmonic traps]
{A separation threshold for ground states of two-component attractive Bose--Einstein condensates in displaced harmonic traps}

\author{Shubin Yu}
\address{Chongqing Key Laboratory of Group \& Graph Theories and Applications, Chongqing University of
Arts and Sciences, Chongqing 402160, People's Republic of China.}
\email{yshubin168@email.swu.edu.cn}

\author{Chun-Lei Tang$^*$}
\address{School of Mathematics and Statistics, Southwest University, Chongqing 400715, People's Republic of China.}
\email{tangcl@swu.edu.cn}

\footnotetext[1]{Corresponding author.}
\footnotetext[2]{Project supported by the National Natural Science Foundation of China (No.12371120).}


\begin{abstract}
In this paper, we study the existence of ground states for two-component attractive Bose--Einstein condensates with displaced harmonic trapping
potentials
\[
V_1(x)=|x-x_1|^2
\ \text{and}\
V_2(x)=|x-x_2|^2,
\]
where $x_1\neq x_2\in\mathbb R^2$.
For intraspecies interactions
 $a_1,a_2\in (0,a^*)$, we focus on the critical interspecies coupling
\[
\beta=\beta^*:=a^*+\sqrt{(a^*-a_1)(a^*-a_2)}
\]
and prove that there exists a separation threshold $d_c\in(0,\infty)$ such that the corresponding constrained minimization problem
admits no minimizer when $0<|x_1-x_2|<d_c$, whereas it admits a minimizer when $|x_1-x_2|\geq d_c$.
The threshold is given by $d_c=\Lambda_*^{-1/4}$, where $\Lambda_*$
is characterized by an auxiliary variational problem and attained.
Our results extend the previous work of Guo et al. [J. Funct. Anal. 276 (2019)] and complete the existence classification of
ground states for two-component attractive Bose--Einstein condensates with harmonic trapping potentials.
\end{abstract}
\maketitle
\vspace {-1cm}

\section{Introduction and main results}

In this paper, we study the following coupled nonlinear Gross--Pitaevskii equations
\begin{equation}\label{eq:intro-system}
\left\{
\begin{aligned}
-\Delta u_1+V_1(x)u_1
&=\mu u_1+a_1u_1^3+\beta u_2^2u_1
\ \text{in }\R^2,
\\
-\Delta u_2+V_2(x)u_2
&=\mu u_2+a_2u_2^3+\beta u_1^2u_2
\ \text{in }\R^2,
\end{aligned}
\right.
\end{equation}
where $(u_1,u_2)\in\mathcal H_1(\R^2)\times\mathcal H_2(\R^2)$ and
\[
\mathcal H_i(\R^2)
:=
\left\{
u\in H^1(\R^2):
\int_{\R^2}V_i(x)|u(x)|^2\dd x<\infty
\right\},\ i=1,2.
\]
The system \eqref{eq:intro-system} arises in the study of two-component
Bose--Einstein condensates with trapping potentials $V_1(x)$ and $V_2(x)$,
see for example
\cite{EsryGreeneBurkeBohn1997,HallMatthewsEnsherWiemanCornell1998,
ParkinsWalls1998,LinWei2006,RoyoLetelier2014}
and the references therein.  Here $\mu\in\R$ is the chemical potential. The parameter $a_i>0$ ($resp. <0$) represents that the intraspecies interaction of the atoms inside
each component is attractive ($resp.$ repulsive), while $\beta>0$  ($resp. <0$) describes
the
interspecies interaction between two components is attractive ($resp.$ repulsive).
We focus on the case where both intraspecies and interspecies interactions are attractive, that is, $a_i>0\ \text{and }  \beta>0,\ i=1,2$.

As shown in \cite[Proposition~A.1]{GuoLiWeiZeng2019}, ground states of
\eqref{eq:intro-system} with $a_i>0\ \text{and }\beta>0$ can be described
equivalently by nonnegative minimizers of the $L^2$-critical constrained
variational problem
\begin{equation}
e(a_1,a_2,\beta)
:=
\inf_{(u_1,u_2)\in\Mcal}
E_{a_1,a_2,\beta}(u_1,u_2),
\ a_1>0,\ a_2>0,\ \beta>0,
\label{eq:intro-minimization}
\end{equation}
where
\[
\Mcal
:=
\left\{
(u_1,u_2)\in\mathcal H_1(\R^2)\times\mathcal H_2(\R^2):
\int_{\R^2}\bigl(u_1^2+u_2^2\bigr)\dd x=1
\right\},
\]
and the Gross--Pitaevskii energy functional is given by
\begin{align}
E_{a_1,a_2,\beta}(u_1,u_2)
&:=
\int_{\R^2}
\left(
|\nabla u_1|^2+|\nabla u_2|^2
+V_1(x)u_1^2+V_2(x)u_2^2
\right)\dd x
\nonumber\\
&\ \ \ -\int_{\R^2}
\left(
\frac{a_1}{2}u_1^4
+\frac{a_2}{2}u_2^4
+\beta u_1^2u_2^2
\right)\dd x .
\label{eq:intro-general-energy}
\end{align}
The trapping potentials
$0\leq V_i(x)\in C_{loc}^{\alpha}(\mathbb R^2)$ ($\alpha\in(0,1)$) are usually assumed to satisfy
\begin{equation}
\lim_{|x|\to\infty}V_i(x)=+\infty,\ \text{both}\
\inf_{x\in\R^2}V_i(x)=0\ \text{and}\ \inf_{x\in\R^2}(V_1(x)+V_2(x))\ \text{are attained}.
\label{eq:intro-potential-assumption}
\end{equation}
Under the assumption \eqref{eq:intro-potential-assumption}, Guo, Li, Wei and Zeng
\cite{GuoLiWeiZeng2019}
obtained a sharp existence and nonexistence result away from the
critical interspecies coupling. To state their result, let $Q$ denote the unique
positive radial solution of
\begin{equation}\label{eqn:Q-equation}
-\Delta Q+Q=Q^3\ \text{in }\R^2,
\end{equation}
see \cite{Kwong1989}.
Set $a^*:=\|Q\|_2^2:=\int_{\mathbb R^2}|Q|^2\dd x$.
For $0<a_1,a_2<a^*$, define the critical coupling
\begin{equation}
\beta^*
:=
a^*+\sqrt{(a^*-a_1)(a^*-a_2)}.
\label{eq:beta-star}
\end{equation}
Then problem \eqref{eq:intro-minimization} admits at least one minimizer
 when
$0<\beta<\beta^*$, whereas it has no minimizer if either $a_1>a^*$,
$a_2>a^*$, or $\beta>\beta^*$.
We mention that this result was also obtained in the earlier work \cite{BaoCai2011}, where the authors
 established general existence and uniqueness results for
ground states of two-component Gross--Pitaevskii systems, including models
with an internal atomic Josephson junction.

The endpoint cases are more delicate.
  Guo, Li, Wei and Zeng
 proved in \cite[Theorem~1.2]{GuoLiWeiZeng2019} that
 if $a_1=a^*$ or $a_2=a^*$, there is no minimizer for problem \eqref{eq:intro-minimization} with $\beta\leq\beta^*$.
 At the critical coupling $\beta=\beta^*$, they further proved that if
$0<a_1,a_2<a^*$ and $V_1(x),V_2(x)$ have at least one common minimum point, then no minimizer
exists.  If $0<a_1,a_2<a^*$ and $V_1(x),V_2(x)$ have no
common minimum point, they proved existence of minimizers under the strict inequality
\begin{equation}
e(a_1,a_2,\beta^*)
<
\frac{1}{
\sqrt{a^*-a_1}+\sqrt{a^*-a_2}}
\inf_{x\in\R^2}
\left[
\sqrt{a^*-a_2}\,V_1(x)
+
\sqrt{a^*-a_1}\,V_2(x)
\right].
\label{eq:intro-GLWZ-condition}
\end{equation}
Therefore, in the
critical coupling case, the relative geometry of the two trapping
potentials
 explicitly affects the existence of minimizers.
However, for a given pair of trapping potentials without a common minimum point,
\eqref{eq:intro-GLWZ-condition}
provides only a sufficient condition for existence.
This naturally raises the question of whether the existence of
minimizers can be characterized completely at the critical coupling when $0<a_1,a_2<a^*$ and the two trapping potentials have no
common
minimum point?

The purpose of the present paper is to answer this question for two harmonic
trapping potentials with distinct centers, that is,
\begin{equation}
V_1(x)=|x-x_1|^2,\ V_2(x)=|x-x_2|^2.
\label{eq:intro-general-harmonic}
\end{equation}
Note that the gradient and nonlinear terms in \eqref{eq:intro-general-energy}
are invariant under a common translation and
rotation.  After translating $x_1$ to the origin and rotating the vector
$x_2-x_1$ onto the first coordinate axis, it suffices to consider
\begin{equation}
V_{1,d}(x):=|x|^2,\
V_{2,d}(x):=|x-de_1|^2,
\label{eq:harmonic-traps}
\end{equation}
where $e_1:=(1,0)\ \text{and}\ d:=|x_1-x_2|>0.$
For these harmonic potentials, it is convenient to work in
\[
\Sigma(\R^2)
:=
\left\{
u\in H^1(\R^2):
\int_{\R^2}|x|^2|u(x)|^2\dd x<\infty
\right\}
\]
and $\mathcal H_i(\R^2)=\Sigma(\R^2)$, $i=1,2$.
Moreover, the constrained
variational problem \eqref{eq:intro-minimization} can be rewritten as
\[
e_*(d)
:=
\inf_{(u_1,u_2)\in\Mcal}
E_{a_1,a_2,\beta^*}^{d}(u_1,u_2),
\]
where
\begin{align*}
E_{a_1,a_2,\beta^*}^{d}(u_1,u_2)&:=\int_{\R^2}
\left(
|\nabla u_1|^2+|\nabla u_2|^2
+V_{1,d}(x)u_1^2+V_{2,d}(x)u_2^2
\right)\dd x\\
&\ \ \ -\int_{\R^2}
\left(
\frac{a_1}{2}u_1^4
+\frac{a_2}{2}u_2^4
+\beta^* u_1^2u_2^2
\right)\dd x .
\end{align*}
Fix $0<a_1,a_2<a^*$ and define
\[
\gamma
:=
\frac{\sqrt{a^*-a_2}}
{\sqrt{a^*-a_1}+\sqrt{a^*-a_2}}
\in(0,1).
\]
From \eqref{eq:intro-GLWZ-condition}, we define
\begin{align*}
W_{\gamma,d}(x)
:&=
\gamma V_{1,d}(x)+(1-\gamma)V_{2,d}(x)\\
&=|x-z_d|^2+\gamma(1-\gamma)d^2,
\end{align*}
where $z_d:=(1-\gamma)de_1$. Then
\[
m_\gamma(d)
:=
\min_{x\in\R^2}W_{\gamma,d}(x)
=
\gamma(1-\gamma)d^2.
\]
Now, we state our first main result.

\begin{theorem}
\label{thm:main}
Let $0<a_1,a_2<a^*$ and $\beta^*$ be given by \eqref{eq:beta-star}.
For the trapping potentials \eqref{eq:harmonic-traps}, there exists a separation threshold
\[
d_c:=\Lambda_*^{-1/4}\in(0,\infty),
\]
depending only on $(a_1,a_2)$, such that
\begin{enumerate}
\item [$(i)$] if $0<d<d_c$, then
\[
e_*(d)=m_\gamma(d),
\]
and $e_*(d)$ is not attained;
\item  [$(ii)$] if $d>d_c$, then
\[
e_*(d)<m_\gamma(d),
\]
and $e_*(d)$ is attained;
\item [$(iii)$] if $d=d_c$,
\[
e_*(d_c)=m_\gamma(d_c).
\]
In particular, $e_*(d_c)$ is attained if and only if the
$\Lambda_*$ defined in \eqref{eq:Lambda-def} is attained.
\end{enumerate}
\end{theorem}
\begin{remark}\rm
Theorem \ref{thm:main} shows that for the displaced harmonic trapping potentials
\eqref{eq:intro-general-harmonic}, there exists
a separation threshold in the distance between the two trap centers that determines the existence or nonexistence of minimizers.
Our method is based
on the variational structure of $\Lambda_*$ together with a subtle scaling argument, see Section \ref{sec:scaling}.
 The main difficulty is to prove
\[
0<\Lambda_*<\infty.
\]
The positivity directly follows from a suitable test function, whereas the finiteness is more delicate.
Using blow-up analysis, an unbounded maximizing sequence,
after rescaling and translation, may approach the manifold
of an optimizer for the
Gagliardo--Nirenberg type inequality \eqref{eq:vector-GN}.
To this end, we establish a quantitative stability estimate near this manifold in
Section~\ref{sec:stability}.
Its proof combines concentration-compactness, the coercivity of the linearized quadratic forms around the optimizer and the standard nondegeneracy of $Q$.
\end{remark}

In what follows, we consider the case $d=d_c$.
By Theorem \ref{thm:main}-($iii$),  it remains to verify that
 $\Lambda_*$ is attained.
A simple compactness argument implies that
 if the maximizing sequence of  $\Lambda_*$ is bounded, then  $\Lambda_*$ is attained.
The main difficulty is therefore to exclude
unbounded maximizing sequences. By  the blow-up analysis, we can show that if the sequence is unbounded, there exists a constant $\Lambda_0\geq0$ defined by \eqref{eq:Lambda-conc} such that  $\Lambda_*\leq \Lambda_0$.
On the other hand, a suitable family of semi-trivial test functions yields the strict inequality $\Lambda_*> \Lambda_0$.
Consequently, $\Lambda_*$ is attained.
The details are presented in Section \ref{sec:endpoint} and
the corresponding result is as follows.
\begin{theorem}\label{thm:main-2}
Under the setting in Theorem {\rm\ref{thm:main}}, $\Lambda_*$ is attained and then
$e_*(d_c)$ is attained.
\end{theorem}

\begin{remark}\rm
We next mention two related problems where the existing results
for the critical coupling cases remain incomplete if the trapping potentials have no common minimum point.
  First,
Guo, Zeng and Zhou \cite{GuoZengZhou2017} considered
 the minimization problem of functional \eqref{eq:intro-general-energy} under the
separate
mass constraints
\[
\|u_1\|_2^2=\|u_2\|_2^2=1.
\]
If $a_1=a_2=a^*-\beta\in (0,a^*)$, they proved in
\cite[Theorem~1.3]{GuoZengZhou2017} that if the trapping
potentials have no common minimum point,
the existence of minimizers follows from a strict binding condition similar to \eqref{eq:intro-GLWZ-condition}.
The recent work \cite{PengZengZhangZhou2026} completes the
classification with respect to the interaction parameters by introducing a
critical surface.
On this
critical surface, the nonexistence result still
requires the  trapping potentials have a common minimum point.
Second, for two-component Bose--Einstein condensates with an internal atomic Josephson
junction, Peng, Wang, Zeng and Zhou \cite{PengWangZengZhou2024} studied the
critical coupling case $\beta=\beta^*$. As shown in \cite[Theorem~1.1]{PengWangZengZhou2024},
if the trapping
potentials have no common minimum point, the existence result is again obtained
 under a strict binding condition similar to \eqref{eq:intro-GLWZ-condition}.
The arguments developed in the present paper may provide some useful ideas for these two problems, which we leave to interested readers.
\end{remark}

\begin{remark}\rm
For the displaced harmonic trapping potentials, the present paper gives a sharp characterization of the existence of ground states at the critical coupling $\beta=\beta^*$.
A natural continuation is to study the limit behavior of minimizers for $d<d_c$ as $\beta\nearrow\beta^*$.
We will investigate this problem in future work, where a new concentration phenomenon, different from \cite{GuoLiWeiZeng2019}, is expected to arise.
\end{remark}

\medskip
The remainder of the paper is organized as follows.   In Section~\ref{sec:prelim}, we give some
preliminaries.   Section~\ref{sec:stability}
 is devoted to accomplishing
 a quantitative stability
estimate.  In Section~\ref{sec:threshold},
we introduce the variational structure of $\Lambda_*$ and verify that
$0<\Lambda_*<\infty$.  In Section~\ref{sec:scaling}, we use the scaling argument to identify the separation threshold and complete
the proof of
Theorem~\ref{thm:main}.  In Section~\ref{sec:endpoint}, we
 prove Theorem~\ref{thm:main-2}.

\section{Preliminaries}
\label{sec:prelim}
In this section, we give
some preliminary results that will be used in the later arguments.
First, from \cite[Lemma A.2]{GuoLiWeiZeng2019},
we introduce the following Gagliardo--Nirenberg type inequality
\begin{equation}\label{eq:vector-GN}
\int_{\R^2}(u_1^2+u_2^2)^2\dd x
\le
\frac{2}{a^*}
\int_{\R^2}(|\nabla u_1|^2+|\nabla u_2|^2)\dd x
\int_{\R^2}(u_1^2+u_2^2)\dd x,
\end{equation}
where $(u_1,u_2)\in H^1(\mathbb R^2)\times H^1(\mathbb R^2)$ and the best constant is attained at $(Q\sin \theta,Q \cos\theta)$ for any $\theta\in [0,2\pi)$.
Moreover, from \cite{Kwong1989}, we have the single component Gagliardo--Nirenberg inequality
\begin{equation}\label{eqn:GN}
\int_{\mathbb R^2}u^4\mathrm dx
\le
\frac2{a^*}
\int_{\mathbb R^2}|\nabla u|^2\mathrm dx
\int_{\mathbb R^2}u^2\mathrm dx,
\end{equation}
where $u\in H^1(\mathbb R^2)$ and the best constant is attained at $Q$.
Recall that $Q$ is the unique positive solution of \eqref{eqn:Q-equation}. Then the Pohozaev identity yields
\begin{equation*}
\int_{\mathbb{R}^{2}}|\nabla Q|^2\mathrm{d}x
=\int_{\mathbb{R}^{2}}Q^2\mathrm{d}x=\frac{1}{2}
\int_{\mathbb{R}^{2}}Q^4\mathrm{d}x,
\end{equation*}
and it follows from \cite[Proposition 4.1]{Gidas-1981} that
\begin{equation}\label{eqn:Q-decay}
|\nabla Q(x)|,\ Q(x)=O(|x|^{-\frac{1}{2}}e^{-|x|})\ \mbox{as}\ |x|\rightarrow\infty.
\end{equation}
Set $q:=\frac{Q}{\sqrt{a^*}}$,
then
\begin{equation}
\|q\|_2^2=\|\nabla q\|_2^2=1,\
\|q\|_4^4=\frac{2}{a^*}\ \text{and}\
-\Delta q+q=a^*q^3.
\label{eq:q-properties}
\end{equation}
By \eqref{eqn:Q-decay}, the following result directly holds.
Meanwhile, we remark that the notation $C$ in the context denotes the positive constants,
which might be changed from line to line and even in the same line.

\begin{lemma}
\label{lem:q-tail}
There exist constants $\sigma>0$, $C>0$ and $R>0$ such that
\begin{equation*}
0<q(x)\le Ce^{-\sigma|x|} \ \text{for}\ |x|>R
\end{equation*}
and
\begin{equation*}
\int_{|x|>R}(1+|x|^2)q^2(x)\dd x
\le Ce^{-\sigma R}.
\end{equation*}
\end{lemma}

In the following, we present some variational decompositions at the critical
coupling.
Set $s_1:=\sqrt{a^*-a_1}\
\text{and}\
s_2:=\sqrt{a^*-a_2},$
then
\begin{equation*}
\gamma=\frac{s_2}{s_1+s_2},
\
1-\gamma=\frac{s_1}{s_1+s_2},
\
\beta^*=a^*+s_1s_2.
\end{equation*}
Define the translation invariant part of the functional $E_{a_1,a_2,\beta^*}^{d}(U)$ with $U=(u_1,u_2)$ by
\begin{align*}
\Fcal(U)
:={}&
\int_{\R^2}(|\nabla u_1|^2+|\nabla u_2|^2)\dd x
-\frac{a_1}{2}\int_{\R^2}u_1^4\dd x
-\frac{a_2}{2}\int_{\R^2}u_2^4\dd x-\beta^*\int_{\R^2}u_1^2u_2^2\dd x,
\end{align*}
then the following decomposition holds.
\begin{lemma}
\label{lem:F-decomp}
For every $U=(u_1,u_2)\in\Mcal$,
\begin{equation*}
\Fcal(U)
=
D(U)
+\frac12\int_{\R^2}
\left(s_1u_1^2-s_2u_2^2\right)^2\dd x,
\end{equation*}
where
\begin{equation*}
D(U)
=
\int_{\R^2}(|\nabla u_1|^2+|\nabla u_2|^2)\dd x
-\frac{a^*}{2}
\int_{\R^2}(u_1^2+u_2^2)^2\dd x.
\end{equation*}
In particular, $\Fcal(U)\ge0$.
\end{lemma}

\begin{proof}
Noting that
\[
a_1=a^*-s_1^2,\
a_2=a^*-s_2^2,\
\beta^*=a^*+s_1s_2,
\]
we have
\[
-\frac{a_1}{2}u_1^4-\frac{a_2}{2}u_2^4-\beta^*u_1^2u_2^2
=
-\frac{a^*}{2}(u_1^2+u_2^2)^2
+\frac12(s_1u_1^2-s_2u_2^2)^2.
\]
Since $U\in\Mcal$, then \eqref{eq:vector-GN} implies that $D(U)\ge0$ and $\Fcal(U)\ge0$.
\end{proof}

For $\eps>0$ and $y\in\R^2$, define
\begin{equation}\label{eqn:Phi-definition}
q_{\eps,y}(x):=\frac1{\eps}q\left(\frac{x-y}{\eps}\right)
\ \text{and}\
\Phi_{\eps,y}
:=
\left(
\sqrt{\gamma} q_{\eps,y},
\sqrt{1-\gamma} q_{\eps,y}
\right).
\end{equation}
\begin{lemma}
\label{lem:F-zero}
Assume $U=(u_1,u_2)\in\Mcal$ and $u_1,u_2\ge0$.
Then $\Fcal(U)=0$ if and only if $U=\Phi_{\eps,y}$ for some
$\eps>0$ and $y\in\R^2$.
\end{lemma}

\begin{proof}
If $U=\Phi_{\eps,y}$, the equality holds in \eqref{eq:vector-GN} and
$s_1u_1^2-s_2u_2^2=0$. By Lemma \ref{lem:F-decomp}, $\Fcal(U)=0$.
Conversely, if $\Fcal(U)=0$, Lemma \ref{lem:F-decomp} gives
\begin{equation}\label{eqn:s-1-s-2}
D(U)=0\ \text{and}\ s_1u_1^2=s_2u_2^2\ \text{a.e. in}\ \mathbb R^2.
\end{equation}
Let $r=(u_1^2+u_2^2)^{1/2}$, then
\begin{equation}\label{eqn:nabla-r}
|\nabla r|^2\le|\nabla u_1|^2+|\nabla u_2|^2.
\end{equation}
Using the Gagliardo--Nirenberg inequality
\eqref{eqn:GN} and $D(U)=0$, we can infer that
the equality holds in \eqref{eqn:nabla-r}
and
$r$ is an optimizer with
$r=q_{\eps,y}$ for some $\eps,y$.
Writing $U=r \omega$ with $|\omega|=1$, we have
$\nabla \omega=0$ and $\omega$ is a constant vector. Thus,  there are $c_1,c_2\ge0$ with $c_1^2+c_2^2=1$ such that
\[
u_1=c_1q_{\eps,y}\ \text{and}\ u_2=c_2q_{\eps,y},
\]
which, together with \eqref{eqn:s-1-s-2}, imply that
$c_1^2=\gamma$ and $c_2^2=1-\gamma$. This completes the proof.
\end{proof}

For unit separation, define the potential energy
\begin{equation*}
\Pcal(U)
:=
\int_{\R^2}|x|^2u_1^2\dd x
+
\int_{\R^2}|x-e_1|^2u_2^2\dd x.
\end{equation*}
Let $z_\gamma:=(1-\gamma)e_1$, $c_\gamma:=\gamma(1-\gamma)$
and
$$
\vartheta
:=
\int_{\R^2}|x|^2q(x)^2\dd x>0,
$$
then the potential energy of $\Phi_{\eps,y}$ is as follows.
\begin{lemma}
\label{lem:bubble-P} There holds
\begin{equation*}
\Pcal(\Phi_{\eps,y})
=
c_\gamma+|y-z_\gamma|^2+\vartheta\eps^2.
\end{equation*}
\end{lemma}

\begin{proof}
Since $q$ is radial, $\int_{\R^2}xq^2 \dd x=0$.  Hence,
\[
\Pcal(\Phi_{\eps,y})
=
\gamma|y|^2+(1-\gamma)|y-e_1|^2+\vartheta\eps^2=
c_\gamma+|y-z_\gamma|^2+\vartheta\eps^2.
\]
\end{proof}

\section{A quantitative stability estimate}
\label{sec:stability}
In this section, we establish a quantitative stability estimate,
which will be used in proving the finiteness of
$\Lambda_*$ in Section~\ref{sec:threshold}.
Define
\begin{equation*}
\Ncal_+
:=
\left\{
U=(u_1,u_2)\in H^1(\R^2)\times H^1(\R^2):
u_i\ge0,\
\|U\|_2^2=1,\
\|\nabla U\|_2^2=1
\right\},
\end{equation*}
where
\[
\|U\|_2^2:=\int_{\R^2}(u_1^2+u_2^2) \dd x,
\
\|\nabla U\|_2^2:=\int_{\R^2}(|\nabla u_1|^2+|\nabla u_2|^2) \dd x.
\]
Let
\begin{equation*}
\Phi(x)
:=
\left(
\sqrt{\gamma} q(x),
\sqrt{1-\gamma} q(x)
\right),
\end{equation*}
then $\Phi(x)$ is an optimizer of \eqref{eq:vector-GN}.
Set manifold $\Zcal:=\{\Phi(\cdot-y):y\in\R^2\}$.
For scalar functions, we use the notation
\[
\langle u,v\rangle
:=
\int_{\R^2} u(x)v(x) \dd x.
\]
For $U=(u_1,u_2)$ and $V=(v_1,v_2)$ in
$L^2(\R^2)\times L^2(\R^2)$, we denote
\[
\langle U,V\rangle_{L^2\times L^2}
:=
\int_{\R^2}\bigl(u_1v_1+u_2v_2\bigr) \dd x,
\]
and
\[
\|U\|_{H^1\times H^1}^2
:=
\|u_1\|_{H^1}^2+\|u_2\|_{H^1}^2
\]
with $\|u\|_{H^1}:=\|u\|_{H^1(\mathbb R^2)}=\left(
\int_{\mathbb R^2}
\bigl(|\nabla u|^2+u^2\bigr)\,\dd x
\right)^{1/2}$.
The distance from $U$ to the manifold $\Zcal$ is
\begin{equation*}
\operatorname{dist}_{H^1\times H^1}(U,\Zcal)
:=
\inf_{y\in\R^2}
\|U-\Phi(\cdot-y)\|_{H^1\times H^1}.
\end{equation*}

\begin{proposition}
\label{prop:stability}
There exist constants $\eta_0>0$ and $C>0$, depending only on
$a_1,a_2$, such that for every $U\in\Ncal_+$ with
$\Fcal(U)\le\eta_0$,
\begin{equation}\label{eqn:quadratic-estimate}
\operatorname{dist}_{H^1\times H^1}(U,\Zcal)^2
\le
C\Fcal(U).
\end{equation}
\end{proposition}

To prove Proposition \ref{prop:stability}, we first establish the following compactness result.

\begin{lemma}
\label{lem:qual-compact}
Let $\{U_n\}\subset\Ncal_+$ and $\Fcal(U_n)\to0$ as $n\to\infty$, then
\[
\operatorname{dist}_{H^1\times H^1}(U_n,\Zcal)\to0\ \text{as}\ n\to\infty.
\]
\end{lemma}

\begin{proof}
By Lemma \ref{lem:F-decomp}, $D(U_n)\to0\ \text{as}\ n\to\infty$. Since
$\|U_n\|_2^2=\|\nabla U_n\|_2^2=1$, then for $U_n=(u_{1n},u_{2n})$,
\begin{equation}
\int_{\R^2}(u_{1n}^2+u_{2n}^2)^2\dd x
\to
\frac{2}{a^*}\ \text{as}\ n\to\infty.
\label{eq:L4-saturation}
\end{equation}
By the
concentration-compactness principle \cite{Lions1984}, we deduce that
 there exist
$\delta>0$, $R>0$ and $y_n\in\R^2$ such that
\begin{equation}
\int_{B_R(y_n)}|U_n|^2\dd x\ge\delta.
\label{eq:nonvanishing}
\end{equation}
Set $V_n(x):=U_n(x+y_n)$. Passing to a subsequence, there exists $V\in H^1(\R^2)\times H^1(\R^2)$ such that as $n\to\infty$,
$
V_n\rightharpoonup V\ \text{in }H^1(\R^2)\times H^1(\R^2),
$
$
V_n\to V\ \text{in }L^2_{\rm loc}(\R^2)\times L^2_{\rm loc}(\R^2)
\ \text{and a.e. in }\R^2.
$
By \eqref{eq:nonvanishing}, $V\not\equiv0$.

Let $W_n:=V_n-V$ and set
$
m:=\|V\|_2^2>0,\  k:=\|\nabla V\|_2^2>0,
$
then we know that
\begin{equation*}
\|V_n\|_2^2
=
m+\|W_n\|_2^2+o(1)\ \text{and}\
\|\nabla V_n\|_2^2
=
k+\|\nabla W_n\|_2^2+o(1),
\end{equation*}
where $o(1)$ denotes $o(1)\to0$ as $n\to\infty$.
This indicates that
\begin{equation}
\|W_n\|_2^2\to1-m
\ \text{and}\
\|\nabla W_n\|_2^2\to1-k\ \text{as}\ n\to\infty.
\label{eq:Hsplit}
\end{equation}
Moreover, the Br\'ezis--Lieb lemma \cite{BrezisLieb1983} gives
\begin{equation}
\int_{\R^2}|V_n|^4 \dd x
=
\int_{\R^2}|V|^4 \dd x
+
\int_{\R^2}|W_n|^4 \dd x
+o(1),
\label{eq:BL}
\end{equation}
where $|V|^2=v_1^2+v_2^2$.
Applying the Gagliardo--Nirenberg inequality \eqref{eq:vector-GN} to $V$ and $W_n$,
\[
\frac{a^*}{2}\int_{\R^2}|V|^4 \dd x\le km
\]
and
\[
\limsup_{n\to\infty}
\frac{a^*}{2}\int_{\R^2}|W_n|^4 \dd x
\le
(1-k)(1-m).
\]
In view of \eqref{eq:L4-saturation} and \eqref{eq:BL}, we conclude that
\begin{equation}\label{eq:no-dichotomy}
1\le
km+(1-k)(1-m)
=1-\bigl[k(1-m)+m(1-k)\bigr]\le1,
\end{equation}
which implies that $m=k=1$.
It follows from \eqref{eq:Hsplit} that
\[
V_n\to V
\ \text{in }H^1(\R^2)\times H^1(\R^2)\ \text{as}\ n\to\infty.
\]
Hence $\Fcal(V)=0$.
By Lemma \ref{lem:F-zero}, we have $V=\Phi(\cdot-y_0)$ for some $y_0\in\mathbb R^2$. This completes the proof.
\end{proof}

By Lemma \ref{lem:qual-compact}, it suffices to establish the desired quadratic estimate \eqref{eqn:quadratic-estimate} near \(\mathcal Z\).
To this end, we
first prove the following translation orthogonality.
For $\Phi(x)=(\phi_1(x),\phi_2(x))$,
define
$$
\pa_j\Phi(x):=\frac{\pa\Phi(x)}{\pa{x_j}}
=(\pa_j\phi_1(x),\pa_j\phi_2(x)).
$$
\begin{lemma}
\label{lem:centering}
For every $\Phi(x-z)\in \Zcal$ with $z\in\mathbb R^2$, there exist uniform constants $\rho_0>0$ and $C>0$ such that if
\[
\|U-\Phi(\cdot-z)\|_{H^1\times H^1}<\rho_0,
\]
then there exists a unique $y=y(U)\in\R^2$ with $|y-z|<C\rho_0$ such that
\begin{equation}
\left\langle U-\Phi(\cdot-y),\pa_j\Phi(\cdot-y)\right\rangle_{L^2\times L^2}
=0,\  j=1,2.
\label{eq:centering-orig}
\end{equation}
Equivalently,
\begin{equation}
\langle \widehat U-\Phi,\pa_j\Phi\rangle_{L^2\times L^2}
=0,\  j=1,2,
\label{eq:centering}
\end{equation}
where $\widehat U(x):=U(x+y)$.
\end{lemma}
\begin{proof}
For any $y=(y_1,y_2)\in\mathbb R^2$, we define $\Phi_y(x):=\Phi(x-y)$.
Consider the map
\[
G:
H^1(\mathbb R^2)\times H^1(\mathbb R^2)
\times \mathbb R^2
\rightarrow \mathbb R^2
\]
defined by
\[
G(U,y)
=
\bigl(G_1(U,y),G_2(U,y)\bigr),
\]
where
\[
G_j(U,y)
=
\left\langle
U-\Phi_y,\partial_j\Phi_y
\right\rangle_{L^2\times L^2},
\qquad j=1,2 .
\]
Since $\Phi$ is smooth and exponentially decaying, the map $G$ is of class
$C^1$.
For any fixed $z\in\mathbb R^2$, we have
\[
G(\Phi_z,z)=0 .
\]
For $j,k=1,2$, differentiating $G_j$ with respect to $y_k$ gives
\[
\partial_{y_k}G_j(U,y)
=
-\left\langle
\partial_{y_k}\Phi_y,\partial_j\Phi_y
\right\rangle_{L^2\times L^2}
+
\left\langle
U-\Phi_y,\partial_{y_k}\partial_j\Phi_y
\right\rangle_{L^2\times L^2} .
\]
Letting $(U,y)=(\Phi_z,z)$,
we obtain
\[
\partial_{y_k}G_j(\Phi_z,z)
=
\left\langle
\partial_k\Phi_z,\partial_j\Phi_z
\right\rangle_{L^2\times L^2}=
\left\langle
\partial_k\Phi,\partial_j\Phi
\right\rangle_{L^2\times L^2}.
\]
Since $q$ is radial, we know that
\[
\langle\pa_1\Phi,\pa_2\Phi\rangle_{L^2\times L^2}=0\ \text{and}\
\|\pa_j\Phi\|_2^2>0.
\]
Then the Jacobian of $G(U,y)$ with respect to $y$ is invertible.
Therefore, the implicit function theorem can be applied at every point
$(\Phi_z,z)$, and there exist
neighborhoods
\(\mathcal U\subset H^1(\R^2)\times H^1(\R^2)\) of \(\Phi_z\), \(\mathcal V\subset\R^2\) of \(z\) and
 a unique
$y=y(U)\in\mathcal V$ such that $G(U,y)=0$ for $U\in\mathcal U$.
Finally, the translation invariance implies that
 the neighborhood size can be chosen uniformly with
respect to $z\in\mathbb R^2$.
Hence, there exist universal constants \(\rho_0>0\) and $C>0$ such that
\eqref{eq:centering-orig} holds if \(
\|U-\Phi(\cdot-z)\|_{H^1\times H^1}<\rho_0.
\)
\end{proof}

Let $\widehat U=U(\cdot+y)$ be the translated function obtained by
Lemma \ref{lem:centering} with $U\in\Ncal_+$ and $\|U-\Phi(\cdot-z)\|_{H^1\times H^1}<\rho_0$.
Let $\widehat U-\Phi=h$, then
\begin{align}
\|h\|_{H^1\times H^1}&\leq \| U(\cdot+y)-\Phi(\cdot+y-z)\|_{H^1\times H^1}+\|\Phi(\cdot+y-z)-\Phi\|_{H^1\times H^1}\notag\\
&\leq \rho_0+C|y-z|\notag\\
&<C\rho_0.
\label{eqn:h-estimate}
\end{align}
Introduce the orthonormal vectors
\begin{equation*}
e=
(\sqrt{\gamma},\sqrt{1-\gamma})\ \text{and}\
f=
(\sqrt{1-\gamma},-\sqrt{\gamma}).
\end{equation*}
Then $h$ can be uniquely
decomposed as $h=e \eta+f \zeta$ with $\eta,\zeta\in H^1(\mathbb R^2)$.
Then $\|h\|_{H^1\times H^1}^2
=
\|\eta\|_{H^1}^2+\|\zeta\|_{H^1}^2$, and for $\widehat U=(\widehat u_1,\widehat u_2)$,
\begin{align}
\widehat u_1=\sqrt{\gamma}(q+\eta)+\sqrt{1-\gamma} \zeta\ \text{and}\
\widehat u_2=\sqrt{1-\gamma}(q+\eta)-\sqrt{\gamma} \zeta.
\label{eq:u1eta}
\end{align}
Furthermore, by the translation invariance and \eqref{eq:centering}, we infer that
\begin{equation}
2\langle q,\eta\rangle
+\|\eta\|_2^2+\|\zeta\|_2^2=0,
\label{eq:mass-exact}
\end{equation}
\begin{equation}
2\langle\nabla q,\nabla\eta\rangle
+\|\nabla\eta\|_2^2+\|\nabla\zeta\|_2^2=0
\label{eq:grad-exact}
\end{equation}
and
\begin{equation}
\langle\eta,\pa_jq\rangle=0,\  j=1,2.
\label{eq:trans-orth}
\end{equation}

Define
\begin{align}
\Qcal_+(\eta)
&:=
\int_{\R^2}
\left(
|\nabla\eta|^2+\eta^2-3Q^2\eta^2
\right)\dd x,
\nonumber\\
\Qcal_\perp(\zeta)
&:=
\int_{\R^2}
\left(
|\nabla\zeta|^2+\zeta^2-\kappa Q^2\zeta^2
\right)\dd x,
\label{eq:Qperp}
\end{align}
where
\begin{equation*}
\kappa:=1-\frac{2s_1s_2}{a^*}\in(-1,1).
\end{equation*}
Now, we expand $\Fcal(U)$ for $\rho:=\|h\|_{H^1\times H^1}\leq 1$.
\begin{lemma}
\label{lem:local-expansion}
Let $\mathcal R(\eta,\zeta):=\Fcal(\widehat U)-
\Qcal_+(\eta)-\Qcal_\perp(\zeta).$
Then for $\rho\le1$, there holds
\begin{equation}
|\mathcal R(\eta,\zeta)|\le C\rho^3.
\label{eq:remainder}
\end{equation}
\end{lemma}

\begin{proof}
Note that
$
\widehat u_1^2+\widehat u_2^2=(q+\eta)^2+\zeta^2
$
and $-\Delta q+q=a^*q^3$, then we obtain
\begin{align*}
D(\widehat U)
&=
-2\int_{\R^2}q\eta \dd x
+\int_{\R^2}\left(|\nabla\eta|^2-3Q^2\eta^2\right)\dd x
+\int_{\R^2}\left(|\nabla\zeta|^2-Q^2\zeta^2\right)\dd x
\nonumber\\
&\ \ \ -2a^*\int_{\R^2}q\eta(\eta^2+\zeta^2) \dd x
-\frac{a^*}{2}\int_{\R^2}(\eta^2+\zeta^2)^2 \dd x.
\end{align*}
From \eqref{eq:u1eta} and \eqref{eq:mass-exact},
\begin{equation*}
s_1\widehat u_1^2-s_2\widehat u_2^2
=
2\sqrt{s_1s_2} (q+\eta)\zeta+(s_1-s_2)\zeta^2
\end{equation*}
and
\[
-2\int_{\R^2}q\eta \dd x=\|\eta\|_2^2+\|\zeta\|_2^2.
\]
Consequently,
\begin{align*}
\frac12\int_{\R^2}(s_1\widehat u_1^2-s_2\widehat u_2^2)^2 \dd x
&=
2s_1s_2\int_{\R^2}(q+\eta)^2\zeta^2 \dd x
\nonumber+
2\sqrt{s_1s_2}(s_1-s_2)\int_{\R^2}(q+\eta)\zeta^3 \dd x
\nonumber\\
&\ \ \ +
\frac{(s_1-s_2)^2}{2}\int_{\R^2}\zeta^4 \dd x.
\end{align*}
Recall that $Q^2=a^*q^2$, then
\[
\int_{\R^2}\left(|\nabla\zeta|^2+\zeta^2-Q^2\zeta^2\right) \dd x
+
2s_1s_2\int_{\R^2}q^2\zeta^2 \dd x
=
\Qcal_\perp(\zeta).
\]
Thus, the direct calculation indicates that all terms of $\mathcal R(\eta,\zeta)$ are cubic or quartic.
Since $q\in L^\infty(\mathbb R^2)$ and $H^1(\R^2)\hookrightarrow L^p(\R^2)$ for every
$2\le p<\infty$, then \eqref{eq:remainder} holds for $\rho\le1$.
\end{proof}

Next, we establish the coercivity of $\Qcal_\perp(\zeta)$ and $\Qcal_+(\eta)$, respectively.
\begin{lemma}
\label{lem:Qperp-coercive}
There exists $c_\perp>0$ such that
\begin{equation}\label{eqn:Qperp-coercive}
\Qcal_\perp(\zeta)
\ge
c_\perp\|\zeta\|_{H^1}^2,\
\forall \zeta\in H^1(\R^2).
\end{equation}
\end{lemma}

\begin{proof}
It suffices first to consider $\zeta\in C_c^\infty(\R^2)$. Let $L_-:=-\Delta+1-Q^2$.
Since $Q>0$, writing $\zeta=Qw$ and using $L_-Q=0$, an integration
by parts gives
\[
\langle L_-\zeta,\zeta\rangle
=
\int_{\R^2}Q^2
\left|\nabla\left(\frac{\zeta}{Q}\right)\right|^2\,\dd x
\ge0.
\]
By density of $C_c^\infty(\R^2)$ in $H^1(\R^2)$ and continuity of
the quadratic form of $L_-$, it follows that
\begin{equation}\label{eqn:L-negative-positive}
\langle L_-\zeta,\zeta\rangle\ge0,\
\zeta\in H^1(\R^2),
\end{equation}
and
\[
\Qcal_\perp(\zeta)
=
\langle L_-\zeta,\zeta\rangle
+(1-\kappa)\int_{\R^2} Q^2\zeta^2 \dd x\geq0
\]
due to $\kappa\in(-1,1)$.
Suppose that \eqref{eqn:Qperp-coercive} does not hold, then there exist $\{\zeta_n\}$ with
$\|\zeta_n\|_{H^1}=1$ and $\Qcal_\perp(\zeta_n)\to0$ as $n\to\infty$.
Then $\int_{\R^2} Q^2\zeta_n^2 \dd x\to0$ as $n\to\infty$. However, it follows from \eqref{eq:Qperp} that
\[
\Qcal_\perp(\zeta_n)=1-\kappa\int_{\R^2} Q^2\zeta_n^2 \dd x\to1\ \text{as}\ n\to\infty,
\]
a contradiction.
\end{proof}

Let
\begin{equation*}
L_+:=-\Delta+1-3Q^2.
\end{equation*}
By \cite{Weinstein1983,Kwong1989}, we know that
\begin{equation}
\ker L_+
=
\operatorname{span}\{\pa_1q,\pa_2q\}.
\label{eq:Lplus-kernel}
\end{equation}
This is the standard nondegeneracy of $Q$.
Define the tangent space
\[
T_\eta:=\{\eta\in H^1(\R^2):\langle q,\eta\rangle=0\}
\]
and the zero set of $\Qcal_+$ on $T_\eta$ by
$$
\mathcal K_+
:=
\{\eta\in T_\eta:\Qcal_+(\eta)=0\}.
$$
\begin{lemma}
\label{lem:Qplus-nonnegative}
For every $\eta\in T_\eta$,
\begin{equation}
\Qcal_+(\eta)\ge0
\label{eq:Qplus-nonnegative}
\end{equation}
and
\begin{equation}
\mathcal K_+
=
\operatorname{span}\{\Lambda q,\pa_1q,\pa_2q\},
\
\Lambda q:=q+x\cdot\nabla q.
\label{eq:kernel-Tm}
\end{equation}
\end{lemma}

\begin{proof}
Fix $\eta\in T_\eta$ and define the normalized path
\begin{equation*}
f_t:=\frac{q+t\eta}{\|q+t\eta\|_2}\ \text{with}\ t\in\mathbb R.
\end{equation*}
Since $\langle q,\eta\rangle=0$,
\[
\|q+t\eta\|_2=(1+t^2\|\eta\|_2^2)^{1/2}.
\]
Applying Taylor's expansion for small $|t|$,
\begin{equation}
f_t
=
q+t\eta-\frac{t^2}{2}\|\eta\|_2^2q+g_t
\label{eq:ft-expansion}
\end{equation}
where $\|g_t\|_{H^1}=O(t^3)$.
Note that
$\|f_t\|_2=1$,
then
the
Gagliardo--Nirenberg inequality \eqref{eqn:GN} gives
\begin{equation*}
\overline D(f_t)
:=
\int_{\R^2}|\nabla f_t|^2 \dd x-\frac{a^*}{2}\int_{\R^2} f_t^4 \dd x
\ge0
\end{equation*}
and $\overline D(q)=0$.
In view of \eqref{eq:q-properties}, we know that for $\varphi\in H^1(\mathbb R^2)$ with $\|\varphi\|_{H^1}$ sufficiently small,
\begin{align}
\overline D(q+\varphi)
={}&
-2\langle q,\varphi\rangle
+\int_{\R^2}\left(|\nabla \varphi|^2-3Q^2\varphi^2\right) \dd x
+O(\|\varphi\|_{H^1}^3).
\label{eq:Ds-local}
\end{align}
Substituting \eqref{eq:ft-expansion} into \eqref{eq:Ds-local} with $\varphi=t\eta-\frac{t^2}{2}\|\eta\|_2^2q+g_t$, we infer that as $t\to0$,
\[
\overline D(f_t)
=
t^2
\left[
\|\nabla\eta\|_2^2+\|\eta\|_2^2
-3\int_{\R^2}Q^2\eta^2 \dd x
\right]
+o(t^2)
=
t^2\Qcal_+(\eta)+o(t^2)\geq0,
\]
which means \eqref{eq:Qplus-nonnegative} holds.

Let $\eta\in T_\eta$ and $\Qcal_+(\eta)=0$.
Since $\Qcal_+$ is nonnegative on $T_\eta$, for any $\xi\in T_\eta$, the quadratic
polynomial $k(t):=\Qcal_+(\eta+t\xi)$ has a minimum at $t=0$. Hence $k'(0)=0$, that is,
\[
\int_{\R^2}\left(
\nabla\eta\cdot\nabla\xi+\eta\xi-3Q^2\eta\xi
\right) \dd x=\langle L_+\eta,\xi\rangle=0,
\ \forall \xi\in T_\eta.
\]
Thus, $L_+\eta=cq$
for some constant $c\in\mathbb R$.
Let $q_\lambda(x)=\lambda q(\lambda x)$.
Differentiating at $\lambda=1$ gives
\[
\left.\frac{\dd}{\dd\lambda}q_\lambda\right|_{\lambda=1}
=
q+x\cdot\nabla q
=
\Lambda q.
\]
Note that $L_+(\Lambda q)=-2q$, then
\[
L_+\left(\eta+\frac{c}{2}\Lambda q\right)=0.
\]
By \eqref{eq:Lplus-kernel}, there exist $b_1,b_2\in\mathbb R$ such that
\[
\eta
=
-\frac{c}{2}\Lambda q+b_1\pa_1q+b_2\pa_2q.
\]
Conversely, it is easy to check that $\Lambda q$ and $\pa_jq$ belong to $T_\eta$, and
$\Qcal_+(\Lambda q)=\Qcal_+(\pa_jq)=0$, $j=1,2$. This proves \eqref{eq:kernel-Tm}.
\end{proof}

Define
\begin{equation}
T:=
\left\{
\eta\in H^1(\R^2):
\langle q,\eta\rangle=0,\
\langle\nabla q,\nabla\eta\rangle=0,\
\langle\eta,\pa_jq\rangle=0,\ j=1,2
\right\}.
\label{eq:T}
\end{equation}

\begin{lemma}
\label{lem:Qplus-coercive}
There exists $c_+>0$ such that
\begin{equation}
\Qcal_+(\eta)\ge c_+\|\eta\|_{H^1}^2,
\ \forall \eta\in T.
\label{eq:Qplus-coercive}
\end{equation}
\end{lemma}

\begin{proof}
Assume by contradiction that there exist $\eta_n\in T$ such that
\[
\|\eta_n\|_{H^1}=1\ \text{and}\
\Qcal_+(\eta_n)\to0\ \text{as}\ n\to\infty.
\]
Since
\[
\Qcal_+(\eta_n)
=
1-3\int_{\R^2}Q^2\eta_n^2 \dd x,
\]
we obtain
\begin{equation}
\int_{\R^2}Q^2\eta_n^2 \dd x\to\frac13\ \text{as}\ n\to\infty.
\label{eq:weighted-limit}
\end{equation}
Passing to a subsequence, there exists $\eta\in H^1(\mathbb R^2)$ such that
\[
\eta_n\rightharpoonup\eta
\ \text{in }H^1(\mathbb R^2)\ \text{as}\ n\to\infty.
\]
By \eqref{eq:weighted-limit},
\[
\int_{\R^2}Q^2\eta^2 \dd x=\frac13,
\]
then $\eta\ne0$.
Meanwhile, we know that $\eta\in T$.
Furthermore,
\[
\Qcal_+(\eta)
=
\|\eta\|_{H^1}^2-3\int_{\R^2}Q^2\eta^2 \dd x
=
\|\eta\|_{H^1}^2-1
\le0.
\]
According to Lemma \ref{lem:Qplus-nonnegative}, we have $\Qcal_+(\eta)\ge0$,  and hence
$\Qcal_+(\eta)=0\ \text{and}\ \|\eta\|_{H^1}=1.$
From \eqref{eq:kernel-Tm}, there are constants $A,B_1,B_2\in\mathbb R$ such that
\[
\eta=A\Lambda q+B_1\pa_1q+B_2\pa_2q.
\]
Recall that $q$ is radial, $\langle\nabla q,\nabla\pa_jq\rangle=0,\  j=1,2.$
On the other hand, we know that
\begin{equation}
\langle q,\Lambda q\rangle=0
\ \text{and }
\langle\nabla q,\nabla\Lambda q\rangle=1,
\label{eq:Lambda-identities}
\end{equation}
where the second identity can be obtained by differentiating
$\|\nabla q_\lambda\|_2^2=\lambda^2$ at $\lambda=1$.
Thus, the gradient condition in $T$ implies that $A=0$.
Finally, since
\[
\langle\pa_1q,\pa_2q\rangle=0\
\mbox{and}\
\|\pa_jq\|_2>0,
\]
we have $B_1=B_2=0$ and then $\eta=0$, a contradiction. This completes the proof.
\end{proof}

\begin{proof}[Proof of Proposition {\rm\ref{prop:stability}}]
 Let $\operatorname{dist}_{H^1\times H^1}(U,\Zcal)<\rho_0$ with $U\in \mathcal N_+$ and $\rho_0$ sufficiently small, then
by the definition of infimum, there is $z\in\mathbb R^2$ such that
$\|U-\Phi(\cdot-z)\|_{H^1\times H^1}<\rho_0.$
Let $\widehat U$ be the translated function obtained by
Lemma \ref{lem:centering}, then \eqref{eq:centering} holds. Moreover, by the argument before Lemma
\ref{lem:local-expansion}, we can write
\[
\widehat U-\Phi=e\eta+f\zeta\ \text{and}\
\rho^2=\|\eta\|_{H^1}^2+\|\zeta\|_{H^1}^2.
\]
From \eqref{eq:mass-exact} and \eqref{eq:grad-exact}, we have
\begin{equation*}
\langle q,\eta\rangle=O(\rho^2)
\ \text{and}\
\langle\nabla q,\nabla\eta\rangle=O(\rho^2).
\end{equation*}
Define
\begin{equation}
\widetilde\eta
:=
\eta-\langle q,\eta\rangle q-(\langle\nabla q,\nabla\eta\rangle-\langle q,\eta\rangle)\Lambda q.
\label{eq:eta-tilde}
\end{equation}
Using \eqref{eq:Lambda-identities}, $
\langle q,\widetilde\eta\rangle=0 \ \text{and}\ \langle\nabla q,\nabla\widetilde\eta\rangle=0.$
Since $q$ and $\Lambda q$ are radial, then we can infer from \eqref{eq:trans-orth} that $\langle\widetilde\eta,\pa_jq\rangle=0.$
Thus, $\widetilde\eta\in T$.
Moreover, there is $C>0$ such that
\begin{equation}
\|\eta-\widetilde\eta\|_{H^1}\le C\rho^2.
\label{eq:eta-tilde-error}
\end{equation}
Write $\eta=\widetilde\eta+\nu$ with $\nu:=\eta-\widetilde\eta$ and $\|\nu\|_{H^1}\le C\rho^2$.
Note that
\[
\Qcal_+(\eta)
=\Qcal_+(\widetilde\eta)+2B(\widetilde\eta,\nu)+\Qcal_+(\nu),
\]
where
\[
B(\widetilde\eta,\nu)
:=
\int_{\R^2}
\left(
\nabla\widetilde\eta\nabla \nu
+\widetilde\eta \nu
-3Q^2\widetilde\eta \nu
\right)\dd x.
\]
Obviously, there is $C>0$ such that
$|B(\widetilde\eta,\nu)|\le C\|\widetilde\eta\|_{H^1}\|\nu\|_{H^1}\le C\rho^3$
and $|\Qcal_+(\nu)|\le C\|\nu\|_{H^1}^2\le C\rho^4$. Hence
$\Qcal_+(\eta)\ge \Qcal_+(\widetilde\eta)-C\rho^3$.
Then Lemma \ref{lem:Qplus-coercive} and \eqref{eq:eta-tilde-error} yield
\begin{align*}
\Qcal_+(\eta)
&\ge
c_+\|\widetilde\eta\|_{H^1}^2-C\rho^3
\nonumber\\
&\ge
c_+\|\eta\|_{H^1}^2-C\rho^3.
\end{align*}
By Lemma \ref{lem:Qperp-coercive},
\[
\Qcal_\perp(\zeta)\ge c_\perp\|\zeta\|_{H^1}^2.
\]
Letting $\rho_0$ sufficiently small, \eqref{eqn:h-estimate} implies that $\rho$ is also sufficiently small.
Using Lemma \ref{lem:local-expansion},
\[
\Fcal(U)
\ge
c_0\rho^2-C\rho^3\ge\frac{c_0}{2}\rho^2\geq
\frac{c_0}{2}\operatorname{dist}_{H^1\times H^1}^2(U,\Zcal),
\]
where $c_0:=\min\{c_+,c_\perp\}>0.$
Finally, by Lemma \ref{lem:qual-compact}, we know that there exists $\eta_0>0$ such that if
$\mathcal F(U)\leq\eta_0$, then $\operatorname{dist}_{H^1\times H^1}(U,\Zcal)<\rho_0$.
The proof is completed.
\end{proof}

\section{An auxiliary variational problem}
\label{sec:threshold}
In this section, we introduce the auxiliary variational problem that will determine the separation threshold in Section \ref{sec:scaling}.
Define
\begin{equation}
\Lambda_*
:=
\sup_{\substack{U\in\Mcal\\ \Pcal(U)<c_\gamma}}
\frac{c_\gamma-\Pcal(U)}{\Fcal(U)}.
\label{eq:Lambda-def}
\end{equation}
By Lemmas \ref{lem:F-zero} and \ref{lem:bubble-P},
if $\Pcal(U)<c_\gamma$, then $\Fcal(U)>0$ and $\Lambda_*\geq0$.
Now we prove that
\begin{equation}
0<\Lambda_*<\infty.
\label{eq:Lambda-positive-finite}
\end{equation}

\begin{lemma}
\label{prop:Lambda-positive}
There holds $\Lambda_*>0$.
\end{lemma}

\begin{proof}
Let
\[
U_\tau(x):=(\tau q(\tau x),0)\ \text{with}\ \tau>0,
\]
then $U_\tau\in\Mcal$ and
$\Pcal(U_\tau)=\frac{\vartheta}{\tau^2}$, which imply that
$\Pcal(U_\tau)<c_\gamma$ for $\tau$ large.
Moreover,
\[
\Fcal(U_\tau)
=
\tau^2-\frac{a_1}{2}\frac{2\tau^2}{a^*}
=
\tau^2\left(1-\frac{a_1}{a^*}\right)>0.
\]
Therefore, $\Lambda_*>0$.
\end{proof}

\begin{proposition}
\label{prop:Lambda-finite}
There holds $\Lambda_*<\infty$.
Equivalently, there exists $C>0$ such that
\begin{equation*}
(c_\gamma-\Pcal(U))_+
\le C\Fcal(U),
\ \forall U\in\Mcal,
\end{equation*}
where $(c_\gamma-\Pcal(U))_+$ is the positive part of $c_\gamma-\Pcal(U)$.
\end{proposition}

\begin{proof}
Assume by contradiction that $\Lambda_*=\infty$,
then there exists $U_n=(u_{1n},u_{2n})\in\Mcal$ such that
\begin{equation}
\Pcal(U_n)<c_\gamma,
\label{eq:P-less-c}
\end{equation}
and
\begin{equation*}
\alpha_n
:=
\frac{\Fcal(U_n)}{\delta_n}
\to0\ \text{as}\ n\to\infty,
\end{equation*}
where $\delta_n:=c_\gamma-\Pcal(U_n)>0.$
Noting that
$\Pcal(u_{1n},u_{2n})=\Pcal(|u_{1n}|,|u_{2n}|)$ and $\Fcal(u_{1n},u_{2n})=\Fcal(|u_{1n}|,|u_{2n}|)$, we may assume $u_{1n},u_{2n}\ge0$.
Since $0<\delta_n\le c_\gamma$,
\begin{equation*}
0\le\Fcal(U_n)=\alpha_n\delta_n
\le c_\gamma\alpha_n\to0\ \text{as}\ n\to\infty.
\end{equation*}
In the following, we use the blow-up analysis to obtain a desired contradiction.

\medskip
\noindent
\textbf{Step 1: the gradient energy blow-up.}

Set
\begin{equation*}
K_n
:=
\int_{\R^2}(|\nabla u_{1n}|^2+|\nabla u_{2n}|^2)\dd x.
\end{equation*}
We claim that $K_n\to\infty$ as $n\to\infty$.
Suppose a subsequence satisfies $K_n\le C$.
From \eqref{eq:P-less-c},
\[
\int_{\R^2}|x|^2u_{1n}^2 \dd x\le c_\gamma.
\]
Since $|x|^2\le2|x-e_1|^2+2$, then
\[
\int_{\R^2}|x|^2u_{2n}^2 \dd x\le2c_\gamma+2.
\]
It follows that
\begin{equation*}
\int_{\R^2}|x|^2(u_{1n}^2+u_{2n}^2) \dd x\le C,
\end{equation*}
which indicates that
\[
\int_{|x|>R}|U_n|^2 \dd x\le CR^{-2}, R>0.
\]
Recall that $U_n\in\Mcal$, then pass to a subsequence if necessary, there exists $U_0\in H^1(\R^2)\times H^1(\R^2)$ such that $\text{as}\ n\to\infty$,
\[
U_n\rightharpoonup U_0\ \text{in }H^1(\R^2)\times H^1(\R^2),
\]
\begin{equation*}
U_n\to U_0
\ \text{in }L^2(\R^2)\times L^2(\R^2) \ \text{and a.e. in }\R^2,
\end{equation*}
which mean that $U_0\in\Mcal$ and $\mathcal F(U_0)\geq0$ by Lemma \ref{lem:F-decomp}. Moreover, the interpolation inequality implies that
\begin{equation*}
U_n\to U_0
\ \text{in }L^4(\R^2)\times L^4(\R^2)\ \text{as}\ n\to\infty.
\end{equation*}
Therefore,
\[
\Fcal(U_0)
\le\liminf_{n\to\infty}\Fcal(U_n)=0,
\]
and then $\Fcal(U_0)=0$. In view of
Lemma \ref{lem:F-zero}, we have $U_0=\Phi_{\eps,y}$ for some
$\eps>0$ and $y\in\R^2$. Then Lemma \ref{lem:bubble-P} gives
$\Pcal(U_0)>c_\gamma$. On the other hand, by  Fatou's lemma, we know that
\[
\Pcal(U_0)\le\liminf_{n\to\infty}\Pcal(U_n)\le c_\gamma,
\]
a contradiction.
Hence $K_n\to\infty\ \text{as}\ n\to\infty$.

\medskip
\noindent
\textbf{Step 2: normalization and stability.}

Define
$\eps_n:=K_n^{-1/2}\to0\ \text{as}\ n\to\infty$
and $\widetilde W_n(x):=\eps_nU_n(\eps_nx).$
Then
$$\|\widetilde W_n\|_2^2=1,\ \|\nabla\widetilde W_n\|_2^2=1\ \text{and}\ \Fcal(\widetilde W_n)=\eps_n^2\Fcal(U_n).$$
Proposition \ref{prop:stability} yields for sufficiently large $n$, there is $y_n\in\R^2$ such that
\begin{equation*}
\|
\widetilde W_n-\Phi(\cdot-y_n)
\|_{H^1\times H^1}
\le
C\eps_n\sqrt{\Fcal(U_n)}.
\end{equation*}
Set
$x_n:=\eps_ny_n\ \text{and}\
W_n(x):=\eps_nU_n(x_n+\eps_nx),$
then
\begin{equation}
\|W_n-\Phi\|_{H^1\times H^1}
\le
C\eps_n\sqrt{\Fcal(U_n)}.
\label{eq:Wn-Phi}
\end{equation}
\medskip
\noindent
\textbf{Step 3: uniform boundedness of $\{x_n\}$.}

Choose $R_0>0$ such that
\[
\int_{B_{R_0}(0)}q^2 \dd x>\frac34.
\]
By \eqref{eq:Wn-Phi}, for sufficiently large $n$,
\[
\int_{B_{R_0}(0)}|W_n|^2 \dd x\ge\frac12,
\]
which means that
\begin{equation}
\int_{B_{\eps_nR_0}(x_n)}|U_n|^2 \dd x\ge\frac12.
\label{eq:mass-core}
\end{equation}
Let $d_n:=\operatorname{dist}(x_n,\{0,e_1\}).$
For $x\in B_{\eps_nR_0}(x_n)$,
\[
\min\{|x|,|x-e_1|\}\ge(d_n-\eps_nR_0)_+.
\]
Since
\[
|x|^2u_{1n}^2+|x-e_1|^2u_{2n}^2
\ge
\min\{|x|^2,|x-e_1|^2\}|U_n|^2,
\]
then
\eqref{eq:mass-core} implies that
\[
\Pcal(U_n)\ge\frac12(d_n-\eps_nR_0)_+^2.
\]
Recall that $\Pcal(U_n)<c_\gamma$, then $\{d_n\}$ is uniformly bounded and there exists $C>0$ such that
\begin{equation}
|x_n|\le C.
\label{eq:xn-bounded}
\end{equation}

\medskip
\noindent
\textbf{Step 4: localized potential energy estimate.}

Let
\[
\xi_1=0,\  \xi_2=e_1\ \text{and}\
A_{i,n}(x):=|x_n+\eps_nx-\xi_i|^2,
\]
then for $W_n(x)=\eps_nU_n(x_n+\eps_nx)=(w_{1n},w_{2n})$,
\begin{equation*}
\Pcal(U_n)
=
\sum_{i=1}^2
\int_{\R^2}A_{i,n}(x)w_{in}(x)^2\dd x.
\label{eq:P-rescaled}
\end{equation*}
By Lemma \ref{lem:q-tail}, there exist $\sigma,C>0$ such that
\[
\int_{|x|>R}(1+|x|^2)q(x)^2\dd x
\le Ce^{-\sigma R}
\]
for large $R>0$.
Choose $L>0$ such that $\sigma L>2$ and set
\begin{equation*}
R_n:=L|\log\eps_n|.
\end{equation*}
Then $R_n\to\infty$ and $\eps_nR_n\to0$.
Using \eqref{eq:xn-bounded}, for sufficiently large $n$,
\begin{equation}
\sup_{|x|\le R_n}A_{i,n}(x)\le C.
\label{eq:A-core-bound}
\end{equation}
Let
\[
\phi_1=\sqrt{\gamma} q,\
\phi_2=\sqrt{1-\gamma} q.
\]
By \eqref{eq:Wn-Phi} and \eqref{eq:A-core-bound}, we infer that
\begin{align*}
\left|
\int_{B_{R_n}(0)}A_{i,n}(w_{in}^2-\phi_i^2) \dd x
\right|
&\le
C\|w_{in}-\phi_i\|_2
(\|w_{in}\|_2+\|\phi_i\|_2)
\nonumber\\
&\le
C\eps_n\sqrt{\Fcal(U_n)}.
\end{align*}
Therefore,
\begin{equation}
\Pcal(U_n)
\ge
\sum_{i=1}^2\int_{B_{R_n}(0)}A_{i,n}\phi_i^2 \dd x
-
C\eps_n\sqrt{\Fcal(U_n)}.
\label{eq:P-after-stability}
\end{equation}
Note that
\begin{align}
\sum_{i=1}^2\int_{\R^2}A_{i,n}\phi_i^2 \dd x
&=
\gamma|x_n|^2+(1-\gamma)|x_n-e_1|^2
+\vartheta\eps_n^2
\nonumber\\
&=
c_\gamma+|x_n-z_\gamma|^2+\vartheta\eps_n^2.
\label{eq:reference-tail-R-2}
\end{align}
Moreover, it follows from Lemma \ref{lem:q-tail} that
\begin{align}
\sum_{i=1}^2\int_{|x|>R_n}A_{i,n}\phi_i^2 \dd x
&\le
C\int_{|x|>R_n}q^2 \dd x
+
C\eps_n^2\int_{|x|>R_n}|x|^2q^2 \dd x
\nonumber\\
&=o(\eps_n^2)\ \text{as}\ n\to\infty.
\label{eq:reference-tail}
\end{align}
Combining \eqref{eq:P-after-stability}-\eqref{eq:reference-tail},
\begin{equation}
\Pcal(U_n)
\ge
c_\gamma+|x_n-z_\gamma|^2+\vartheta\eps_n^2
-C\eps_n\sqrt{\Fcal(U_n)}
-o(\eps_n^2) \ \text{as}\ n\to\infty.
\label{eq:key-P-lower}
\end{equation}

\medskip
\noindent
\textbf{Step 5: contradiction.}

Since $\Pcal(U_n)=c_\gamma-\delta_n$, then
\eqref{eq:key-P-lower} gives
\[
\delta_n+|x_n-z_\gamma|^2+\vartheta\eps_n^2
\le
C\eps_n\sqrt{\Fcal(U_n)}+o(\eps_n^2).
\]
For sufficiently large $n$, we have
\begin{equation*}
\delta_n+\frac{\vartheta}{2}\eps_n^2
\le
C\eps_n\sqrt{\Fcal(U_n)}.
\end{equation*}
Recall that $\Fcal(U_n)=\alpha_n\delta_n$,
\[
\delta_n+\frac{\vartheta}{2}\eps_n^2
\le
C\sqrt{\alpha_n}\eps_n\sqrt{\delta_n}.
\]
Set $X_n:=\sqrt{\delta_n}/\eps_n$, then
\begin{equation*}
X_n^2+\frac{\vartheta}{2}
\le
C\sqrt{\alpha_n}X_n\le
\frac12X_n^2+\frac{C^2}{2}\alpha_n,
\end{equation*}
which means that
$\frac12X_n^2+\frac{\vartheta}{2}
\le
\frac{C^2}{2}\alpha_n$
and $\vartheta\le C^2\alpha_n\to0$, contradicting $\vartheta>0$.
\end{proof}

\section{Scaling argument and proof of Theorem \ref{thm:main}}
\label{sec:scaling}

In this section, we exploit the scaling argument
 to
translate the variational structure of $\Lambda_*$ into the distance $d$, and prove
Theorem~\ref{thm:main}.
For $d>0$, define
\begin{equation*}
\Pcal_d(U)
=
\int_{\R^2}|x|^2u_1^2 \dd x
+
\int_{\R^2}|x-de_1|^2u_2^2 \dd x,
\end{equation*}
then
\begin{equation*}
E_{a_1,a_2,\beta^*}^{d}(U)
=
\Fcal(U)+\Pcal_d(U).
\end{equation*}
For $U=(u_1,u_2)\in\Mcal$, define $T_dU:=((T_dU)_1,(T_dU)_2)$ with
\begin{equation*}
(T_dU)_i(x)
:=
\frac1d u_i\left(\frac{x}{d}\right), i=1,2.
\end{equation*}
Then $T_d$ is a bijection of $\Sigma(\R^2)\times\Sigma(\R^2)$ and preserves the mass.
Moreover,
\begin{equation*}
\Fcal(T_dU)=d^{-2}\Fcal(U)\ \text{and}\
\Pcal_d(T_dU)=d^2\Pcal(U).
\end{equation*}
Therefore,
\begin{equation}
e_*(d)=d^2H(d^{-4}),
\label{eq:e-H}
\end{equation}
where
\begin{equation*}
H(\lambda)
:=
\inf_{U\in\Mcal}
\left\{
\Pcal(U)+\lambda\Fcal(U)
\right\},
\  \lambda>0.
\end{equation*}

\begin{lemma}
\label{lem:H-upper}
For every $\lambda>0$, there holds
\begin{equation*}
H(\lambda)\le c_\gamma.
\end{equation*}
Equivalently, $e_*(d)\leq c_\gamma d^2=m_\gamma (d)$.
\end{lemma}

\begin{proof}

Take the test function $\Phi_{\eps,y}$ as in \eqref{eqn:Phi-definition} with $y=z_\gamma$.
By Lemmas \ref{lem:F-zero} and \ref{lem:bubble-P},
\[
\Fcal(\Phi_{\eps,z_\gamma})=0
\ \text{and}\
\Pcal(\Phi_{\eps,z_\gamma})
=
c_\gamma+\vartheta\eps^2.
\]
Letting $\eps\to0$, we get that $H(\lambda)\le c_\gamma.$
\end{proof}

\begin{proposition}
\label{prop:H-transition}
Let $\Lambda_*$ be defined by \eqref{eq:Lambda-def}.
\begin{enumerate}
\item [$(i)$]If $\lambda>\Lambda_*$, then
\[
H(\lambda)=c_\gamma,
\]
and $H(\lambda)$ is not attained.
\item [$(ii)$]If $0<\lambda<\Lambda_*$, then
\[
H(\lambda)<c_\gamma.
\]
\item [$(iii)$]If $\lambda=\Lambda_*$, then
\[
H(\Lambda_*)=c_\gamma.
\]
Moreover, $H(\Lambda_*)$ is attained if and only if $\Lambda_*$ is attained.
\end{enumerate}
\end{proposition}
\begin{proof}
Let $U\in\Mcal$.
If $\Pcal(U)\ge c_\gamma$, then $\Pcal(U)+\lambda\Fcal(U)\ge c_\gamma$
because $\Fcal(U)\ge0$.
If $\Pcal(U)<c_\gamma$,
by the definition of $\Lambda_*$, we know that
\begin{equation}
c_\gamma-\Pcal(U)
\le
\Lambda_*\Fcal(U).
\label{eq:Lambda-ineq}
\end{equation}

If $\lambda>\Lambda_*$, then for $\Pcal(U)<c_\gamma$,
\[
\Pcal(U)+\lambda\Fcal(U)
>
\Pcal(U)+\Lambda_*\Fcal(U)
\ge c_\gamma.
\]
Thus, $H(\lambda)\geq c_\gamma$, and Lemma \ref{lem:H-upper} gives
\[
H(\lambda)=c_\gamma.
\]
If there exists a minimizer, then either $\Pcal(U)<c_\gamma$, which is impossible,
or $\Pcal(U)\ge c_\gamma$, which means
\[
\Pcal(U)=c_\gamma\ \text{and}\ \Fcal(U)=0.
\]
This contradicts Lemma \ref{lem:bubble-P}. Thus, the infimum $H(\lambda)$ is not attained.

If $\lambda<\Lambda_*$, the definition of the supremum implies that there is
$U\in\Mcal$ such that $\Pcal(U)<c_\gamma$ and
\[
\frac{c_\gamma-\Pcal(U)}{\Fcal(U)}>\lambda.
\]
Equivalently,
\[
\Pcal(U)+\lambda\Fcal(U)<c_\gamma,
\]
then
 $H(\lambda)<c_\gamma$.

If $\lambda=\Lambda_*$, then by \eqref{eq:Lambda-ineq}, we have
\[
\Pcal(U)+\Lambda_*\Fcal(U)\ge c_\gamma
\]
for all $U\in\Mcal$. Together with Lemma \ref{lem:H-upper}, $H(\Lambda_*)=c_\gamma.$
If the supremum in \eqref{eq:Lambda-def} is attained at $U$, then
\[
\Pcal(U)+\Lambda_*\Fcal(U)=c_\gamma,
\]
that is, $U$ is a minimizer of $H(\Lambda_*)$.
Conversely, suppose $U$ is a minimizer of $H(\Lambda_*)$,
then
\[
\Pcal(U)+\Lambda_*\Fcal(U)=c_\gamma.
\]
If $\Pcal(U)\ge c_\gamma$, then
$\Pcal(U)=c_\gamma$ and $\Fcal(U)=0$, contradicting
Lemma \ref{lem:bubble-P}. Hence, $\Pcal(U)<c_\gamma$ and
\[
\Lambda_*
=
\frac{c_\gamma-\Pcal(U)}{\Fcal(U)},
\]
which mean that $\Lambda_*$ is attained.
\end{proof}
\noindent\textbf{Proof of Theorem \ref{thm:main}.}
By \eqref{eq:Lambda-positive-finite}, define
\[
d_c:=\Lambda_*^{-1/4}\in(0,\infty).
\]
Recall that
\[
e_*(d)=d^2H(d^{-4})
\ \text{and}\
m_\gamma(d)=c_\gamma d^2.
\]

If $0<d<d_c$, then $d^{-4}>\Lambda_*$.
By Proposition \ref{prop:H-transition}-$(i)$, we know that $H(d^{-4})=c_\gamma$
and $H(d^{-4})$ is not attained. Equivalently,
$e_*(d)=c_\gamma d^2=m_\gamma(d)$ is not attained.

If $d>d_c$, then $d^{-4}<\Lambda_*$. By Proposition \ref{prop:H-transition}-$(ii)$, $H(d^{-4})<c_\gamma,$
then
\begin{equation}\label{eq:strict-binding}
e_*(d)=d^2H(d^{-4})<c_\gamma d^2=m_\gamma(d).
\end{equation}
By \cite[Theorem 1.2]{GuoLiWeiZeng2019}, we know that
there exists a minimizer of
$e_*(d)$. Here we give a short proof.
Let $\{U_n\}\subset\Mcal$ be a minimizing sequence of $e_*(d)$, then
$\Fcal(U_n)$ and $\Pcal_d(U_n)$ are bounded.
Let $K_n=\|\nabla U_n\|_2^2.$
 Note that
\[
|x|^2\le2|x-de_1|^2+2d^2.
\]
 If $\{K_n\}$ is bounded, then
similar to the argument
in Step 1 of the proof of Proposition \ref{prop:Lambda-finite},
we know that  $e_*(d)$ is attained.
If $\{K_n\}$ is unbounded, after passing to a subsequence, we may assume that $K_n\to\infty\ \text{as}\ n\to\infty$.
Then using the same notations as in Steps 2-4 of the proof of Proposition \ref{prop:Lambda-finite} and replacing $\xi_2=e_1$ with $\xi_2=de_1$, we can obtain that
for $\eps_n=K_n^{-1/2}$,
\[
\begin{aligned}
\Pcal_d(U_n)
\ge
m_\gamma(d)+|x_n-z_d|^2+\vartheta\eps_n^2-
C\eps_n\sqrt{\Fcal(U_n)}-o(\eps_n^2)\ \text{as}\ n\to\infty,
\end{aligned}
\]
where $z_d=(1-\gamma)de_1.$
Consequently,
\[
\liminf_{n\to\infty}\Pcal_d(U_n)\ge m_\gamma(d).
\]
Since $\Fcal(U_n)\ge0$,
\[
e_*(d)
=
\lim_{n\to\infty}\left[\Fcal(U_n)+\Pcal_d(U_n)\right]
\ge m_\gamma(d),
\]
contradicting \eqref{eq:strict-binding}.
Hence, $\{K_n\}$ is bounded and $e_*(d)$ is attained.

If $d=d_c$, then $d_c^{-4}=\Lambda_*$. By
Proposition \ref{prop:H-transition}-$(iii)$, we have
\[
e_*(d_c)=c_\gamma d_c^2=m_\gamma(d_c).
\]
Moreover, $e_*(d_c)$ is  attained if and only if $\Lambda_*$ is attained.
\qed

\section{Proof of Theorem \ref{thm:main-2}}
\label{sec:endpoint}
In this section, we study the case $d=d_c$ and prove Theorem \ref{thm:main-2}.
By Theorem~\ref{thm:main}-$(iii)$, this is equivalent to proving that
$\Lambda_*$ is  attained.
Let
\begin{equation*}
L_\perp:=-\Delta+1-\kappa Q^2,
\
\kappa=1-\frac{2s_1s_2}{a^*}\in(-1,1).
\end{equation*}
By Lemma \ref{lem:Qperp-coercive}, the quadratic form
\[
\Qcal_\perp(\zeta)=\langle L_\perp\zeta,\zeta\rangle
\]
is coercive on $H^1(\R^2)$.  Hence, by the Lax-Milgram theorem, there is a
unique $\psi_\kappa\in H^1(\R^2)$ satisfying
\begin{equation}
L_\perp\psi_\kappa=q
\ \text{in }H^{-1}(\R^2).
\label{eq:psi-kappa}
\end{equation}
Define
\begin{equation*}
J_\kappa
:=
\langle q,\psi_\kappa\rangle
=
\Qcal_\perp(\psi_\kappa)
>0.
\end{equation*}
Equivalently,
\begin{equation}
J_\kappa
=
\sup_{\zeta\in H^1(\R^2)\setminus\{0\}}
\frac{\left(\int_{\R^2}q\zeta \dd x\right)^2}
{\Qcal_\perp(\zeta)}.
\label{eq:J-variational}
\end{equation}
Indeed, using  \eqref{eq:psi-kappa} and the
Cauchy-Schwarz inequality in the inner product induced by $\Qcal_\perp$, we can infer that
\[
\left|\int_{\R^2}q\zeta \dd x\right|^2
=
|\langle L_\perp\psi_\kappa,\zeta\rangle|^2
\le
\Qcal_\perp(\psi_\kappa)\Qcal_\perp(\zeta)
=
J_\kappa\Qcal_\perp(\zeta),
\]
and equality is obtained for $\zeta=\psi_\kappa$. Thus, \eqref{eq:J-variational} holds.

Recall
\[
c_\gamma=\gamma(1-\gamma),\
\vartheta=\int_{\R^2}|x|^2q^2 \dd x.
\]
and
define the level
\begin{equation}
\Lambda_0:=
\frac{c_\gamma(1-2\gamma)^2}{\vartheta}J_\kappa.
\label{eq:Lambda-conc}
\end{equation}
By the above argument, this is an exact constant.
We will prove that if the maximizing sequence of $\Lambda_*$ is bounded, then
$\Lambda_*$ is  attained. If the maximizing sequence is unbounded, then $\Lambda_*\leq\Lambda_0.$
\begin{lemma}
\label{lem:bounded-max-seq}
Let $\{U_n\}\subset\Mcal$ satisfy
\begin{equation}\label{eq:max-seq}
\Pcal(U_n)<c_\gamma\ \text{and}\
\frac{c_\gamma-\Pcal(U_n)}{\Fcal(U_n)}
\longrightarrow\Lambda_* \ \text{as}\ n\to\infty.
\end{equation}
If
$\sup_n\|\nabla U_n\|_2<\infty,$
then $\Lambda_*$ is attained.
\end{lemma}

\begin{proof}
We may assume $U_n=(u_{1n},u_{2n})$ with $u_{1n},u_{2n}\geq0$.  Since $\Pcal(U_n)<c_\gamma$,
the same argument as in Step 1 of the proof of
Proposition \ref{prop:Lambda-finite} gives
\[
\sup_n\int_{\R^2}|x|^2|U_n|^2 \dd x<\infty.
\]
Moreover, passing to a subsequence, there is $U\in H^1(\R^2)\times H^1(\R^2)$ such that as  $n\to\infty$,
\[
U_n\rightharpoonup U\ \text{in }H^1(\R^2)\times H^1(\R^2),
\]
and
\[
U_n\to U\ \text{in }L^2(\R^2)\times L^2(\R^2)\ \text{and in }L^4(\R^2)\times L^4(\R^2).
\]
Thus, we have $U\in\Mcal$.
Passing to a further subsequence, we suppose that
\[
\Pcal(U_n)\to p\ \text{and}\
\Fcal(U_n)\to f\ \text{as}\ n\to\infty,
\]
since both sequences are bounded.
By the weak lower semicontinuity of norm and Fatou's lemma, we know that
\[
\Pcal(U)\le p\leq c_\gamma\ \text{and}\
\Fcal(U)\le f.
\]
Now, we claim that $c_\gamma-p>0$ and $f>0$.
If $c_\gamma-p=0$, then
\eqref{eq:max-seq} and $\Lambda_*>0$ imply $\Fcal(U_n)\to0$ and
$f=0$.  Hence, $\Fcal(U)=0$ and $\Pcal(U)\le c_\gamma$,
contradicting Lemmas \ref{lem:F-zero} and \ref{lem:bubble-P}.
The same argument excludes $f=0$.
It follows that $\Pcal(U)<c_\gamma$, $\Fcal(U)>0$ and
\[
\frac{c_\gamma-\Pcal(U)}{\Fcal(U)}
\ge
\frac{c_\gamma-p}{f}
=
\Lambda_*.
\]
By the definition of $\Lambda_*$, we conclude that $U$ is a maximizer of $\Lambda_*$.
\end{proof}

\begin{lemma}
\label{thm:concentration-bound}
Let $\{U_n\}\subset\Mcal$ satisfy \eqref{eq:max-seq}.
If $\|\nabla U_n\|_2^2\to\infty$ as $n\to\infty$,
then
\begin{equation}
\Lambda_*
\le
\Lambda_0.
\label{eq:Lambda-upper-conc}
\end{equation}
\end{lemma}

\begin{proof}
Let
\[
\delta_n=c_\gamma-\Pcal(U_n)>0\ \text{and}\
F_n:=\Fcal(U_n).
\]
By \eqref{eq:max-seq},
\begin{equation}
\frac{\delta_n}{F_n}\to\Lambda_*\in(0,\infty) \ \text{as}\ n\to\infty.
\label{eq:ratio-endpoint}
\end{equation}
Since $0<\delta_n\le c_\gamma$, this implies that $\sup_n F_n<\infty.$
Now, we use the same notations as in Steps 2-4 of the proof of Proposition \ref{prop:Lambda-finite}.
Set $K_n=\|\nabla U_n\|_2^2$ and $\eps_n=K_n^{-1/2}\to0$, then
$\Fcal(\eps_n U_n(\eps_n\cdot))=\eps_n^2F_n\to0\ \text{as}\ n\to\infty$.
As in Step 2, there are $x_n$ and $W_n(x)=\eps_nU_n(x_n+\eps_nx)$
such that for sufficiently large $n$,
\begin{equation}
\|W_n-\Phi\|_{H^1\times H^1}
\le
C\eps_n\sqrt{F_n}.
\label{eq:endpoint-stability}
\end{equation}
Moreover, the estimate
\eqref{eq:key-P-lower} holds, that is,
\begin{equation}
\delta_n+|x_n-z_\gamma|^2+\frac{\vartheta}{2}\eps_n^2
\le
C\eps_n\sqrt{F_n}
\label{eq:endpoint-basic}
\end{equation}
for sufficiently large $n$.
By \eqref{eq:ratio-endpoint}, for sufficiently large $n$,
\[
\frac{\Lambda_*}{2}
\le
\frac{\delta_n}{F_n}
\le
2\Lambda_*.
\]
Putting $X_n=\sqrt{\delta_n}/\eps_n$ in
\eqref{eq:endpoint-basic}, for sufficiently large $n$, we have
\[
X_n^2+\frac{\vartheta}{2}
\le
CX_n
\]
with $C>0$ independent of $n$.  Hence, there are constants $C_1,C_2>0$ such that for sufficiently large $n$,
\begin{equation*}
0<C_1\le X_n\le C_2.
\end{equation*}
Consequently,
\begin{equation}\label{eq:X-bounds-endpoint}
C_1^2\eps_n^2\leq\delta_n\leq C_2^2\eps_n^2,\
\frac{C_1^2}{2\Lambda_*}\eps_n^2\leq F_n\leq\frac{2C_2^2}{\Lambda_*}\eps_n^2
\end{equation}
and
\begin{equation}\label{eqn:x-n-z}
|x_n-z_\gamma|=O(\eps_n).
\end{equation}
By
\eqref{eq:endpoint-stability}, we infer that
\begin{equation*}
\|W_n-\Phi\|_{H^1\times H^1}=O(\eps_n^2).
\end{equation*}
Thus, for sufficiently large $n$, Lemma \ref{lem:centering} is available.
Similar to the argument before Lemma
\ref{lem:local-expansion},
we may write $\widehat W_n:=W_n(\cdot+z_n)$ with $|z_n|=O(\eps_n^2)$ and decompose
\begin{equation}
\widehat W_n=\Phi+\eps_n^2h_n,
\
h_n=e\eta_n+f\zeta_n,
\
\|h_n\|_{H^1\times H^1}\le C.
\label{eq:h-n}
\end{equation}
Moreover,
\begin{align}
\langle q,\eta_n\rangle
=
-\frac{\eps_n^2}{2}
\left(\|\eta_n\|_2^2+\|\zeta_n\|_2^2\right),
\label{eq:eta-mass-small}
\end{align}
\begin{align}
\langle\nabla q,\nabla\eta_n\rangle
=
-\frac{\eps_n^2}{2}
\left(\|\nabla\eta_n\|_2^2+\|\nabla\zeta_n\|_2^2\right)
\label{eq:eta-grad-small}
\end{align}
and
\begin{equation}
\langle\eta_n,\pa_jq\rangle=0,\  j=1,2.
\label{eq:eta-trans-small}
\end{equation}
Passing to a subsequence, there exist $\eta,\zeta\in H^1(\R^2)$ and $\xi\in\R^2$ such that as $n\to\infty$,
\begin{equation}\label{eqn:zeta-weak}
\eta_n\rightharpoonup\eta,\
\zeta_n\rightharpoonup\zeta
\ \text{in }H^1(\R^2),
\end{equation}
and
\begin{equation}\label{eqn:xi-n-definition}
\xi_n:=\frac{x_n-z_\gamma}{\eps_n}\to\xi\in\R^2.
\end{equation}
Then \eqref{eq:eta-mass-small}-\eqref{eq:eta-trans-small} imply that
$\eta\in T$, where $T$ is defined in \eqref{eq:T}.
Moreover, by \eqref{eq:Qplus-nonnegative}, $\Qcal_+(\eta)\ge0$.

In the following, we establish the fine estimates for $F_n$ and $\delta_n$, respectively.
 Applying Lemma
\ref{lem:local-expansion} to $\widehat W_n=\Phi+\eps_n^2h_n$, we can deduce that
\begin{equation*}
\frac{F_n}{\eps_n^2}=\frac{\Fcal(\widehat W_n)}{\eps_n^4}
=
\Qcal_+(\eta_n)+\Qcal_\perp(\zeta_n)+o(1).
\end{equation*}
Using weak lower semicontinuity and \eqref{eqn:zeta-weak}, we therefore obtain
\begin{equation}
\liminf_{n\to\infty}\frac{F_n}{\eps_n^2}
\ge
\Qcal_+(\eta)+\Qcal_\perp(\zeta).
\label{eq:den-liminf}
\end{equation}
On the other hand,  recall from \eqref{eq:h-n} that
\[
W_n(x)=\widehat W_n(x-z_n)
=\Phi(x-z_n)+\varepsilon_n^2 h_n(x-z_n).
\]
For $W_n(x)=(w_{1n},w_{2n})$, $h_n=(h_{1n},h_{2n})$, $\phi_1=\sqrt{\gamma} q$ and $\phi_2=\sqrt{1-\gamma} q$, we know that
\[
w_{in}(x)=\phi_i(x-z_n)+\varepsilon_n^2 h_{in}(x-z_n),
\qquad i=1,2.
\]
Choose $R_n=L|\log\eps_n|$ as in the proof of Proposition
\ref{prop:Lambda-finite}.
Let
\[
\widehat A_{1,n}(x)=|x_n+\eps_nx+\eps_n z_n|^2,\
\widehat A_{2,n}(x)=|x_n+\eps_nx-e_1+\eps_n z_n|^2.
\]
By Lemma \ref{lem:q-tail}, we infer that
\begin{align}
\Pcal(U_n)
\ge{}&
\sum_{i=1}^2\int_{B_{R_n}(0)}\widehat A_{i,n}\phi_i^2 \dd x
+
2\eps_n^2
\sum_{i=1}^2\int_{B_{R_n}(0)}\widehat A_{i,n}\phi_i h_{in} \dd x
+o(\eps_n^2) \ \text{as}\ n\to\infty.
\label{eq:P-second-order-core}
\end{align}
Recall that $|z_n|=O(\varepsilon_n^2)$, then by the definition of $\xi_n$ (see \eqref{eqn:xi-n-definition}),
$$
\sum_{i=1}^2\int_{B_{R_n}(0)}\widehat A_{i,n}\phi_i^2 \dd x=c_\gamma+\eps_n^2\left(\vartheta+|\xi_n|^2\right)
+o(\eps_n^2) \ \text{as}\ n\to\infty.
$$
On the other hand, by \eqref{eqn:x-n-z} and $|z_n|=O(\eps_n^2)$, we have
\[
\widehat A_{1,n}(x)=(1-\gamma)^2+o(1)\ \text{and}\
\widehat A_{2,n}(x)=\gamma^2+o(1)
\]
uniformly for $x\in B_{R_n}(0)$.
Hence
\begin{align}
\sum_{i=1}^2\int_{B_{R_n}(0)}\widehat A_{i,n}\phi_i h_{in} \dd x
={}&
c_\gamma\int_{\R^2}q\eta_n \dd x
+
\sqrt{c_\gamma}(1-2\gamma)\int_{\R^2}q\zeta_n \dd x
+o(1).
\label{eq:linear-potential}
\end{align}
Combining \eqref{eq:eta-mass-small}, \eqref{eqn:zeta-weak}
 and \eqref{eq:P-second-order-core}-\eqref{eq:linear-potential},
we conclude that
$$
\Pcal(U_n)\geq c_\gamma+\eps_n^2\left(\vartheta+|\xi_n|^2\right)
+2\sqrt{c_\gamma}(1-2\gamma)\varepsilon_n^2\int_{\R^2}q\zeta \dd x+o(\varepsilon_n^2)\ \text{as}\ n\to\infty,
$$
which implies that
\begin{equation}
\limsup_{n\to\infty}\frac{\delta_n}{\eps_n^2}
\le
-\vartheta-|\xi|^2- 2\sqrt{c_\gamma} (1-2\gamma)\int_{\R^2}q\zeta \dd x.
\label{eq:num-limsup}
\end{equation}
Together with \eqref{eq:X-bounds-endpoint}, we have
\begin{equation*}
-\vartheta-|\xi|^2- 2\sqrt{c_\gamma} (1-2\gamma)\int_{\R^2}q\zeta \dd x>0,
\end{equation*}
which signifies that $\zeta\ne0$. Recall that $\Qcal_+(\eta)\ge0$, then we can infer from  Lemma \ref{lem:Qperp-coercive} that
\[
\Qcal_+(\eta)+\Qcal_\perp(\zeta)
\ge
\Qcal_\perp(\zeta)>0.
\]

Finally, passing to a further subsequence if necessary, we may assume that
\[
\frac{\delta_n}{\eps_n^2}\to A_0>0 \ \text{as}\ n\to\infty.
\]
Then \eqref{eq:ratio-endpoint} implies
\[
\frac{F_n}{\eps_n^2}\to \frac{A_0}{\Lambda_*}>0 \ \text{as}\ n\to\infty.
\]
By \eqref{eq:den-liminf} and \eqref{eq:num-limsup}, we have
\[
\frac{A_0}{\Lambda_*}\ge \Qcal_+(\eta)+\Qcal_\perp(\zeta)\ge
\Qcal_\perp(\zeta)>0
\]
and
\[
A_0\le -\vartheta-|\xi|^2- 2\sqrt{c_\gamma} (1-2\gamma)\int_{\R^2}q\zeta \dd x.
\]
Therefore,
\begin{align*}
\Lambda_*=A_0\times \frac{\Lambda_*}{A_0}
&\le
\frac{
-\vartheta-|\xi|^2- 2\sqrt{c_\gamma} (1-2\gamma)\int_{\R^2}q\zeta \dd x
}{
\Qcal_+(\eta)+\Qcal_\perp(\zeta)}
\nonumber\\
&\le
\frac{-\vartheta- 2\sqrt{c_\gamma} (1-2\gamma)\int_{\R^2}q\zeta \dd x}
{\Qcal_\perp(\zeta)}.
\end{align*}
Recall from \eqref{eq:J-variational} that
\[
\left|\int_{\R^2}q\zeta \dd x\right|
\le
\sqrt{J_\kappa} \Qcal_\perp(\zeta)^{1/2},
\]
then
\[
\frac{-\vartheta- 2\sqrt{c_\gamma} (1-2\gamma)\int_{\R^2}q\zeta \dd x}{\Qcal_\perp(\zeta)}
\le
-\frac{\vartheta}{\Qcal_\perp(\zeta)}
+
\frac{| 2\sqrt{c_\gamma} (1-2\gamma)|\sqrt{J_\kappa}}{(\Qcal_\perp(\zeta))^{1/2}}.
\]
Consider the function $$
\Gamma(t):=-\frac{\vartheta}{t^2}
+
\frac{| 2\sqrt{c_\gamma} (1-2\gamma)|\sqrt{J_\kappa}}{t}$$
for $t>0$.
If $ 2\sqrt{c_\gamma} (1-2\gamma)\ne0$,
the function  $\Gamma(t)$ reaches its maximum at
\[
t_0=\frac{\vartheta}{| \sqrt{c_\gamma} (1-2\gamma)|\sqrt{J_\kappa}}
\]
and
\[
\Gamma(t_0)=\frac{ c_\gamma (1-2\gamma)^2J_\kappa}{\vartheta}
=
\Lambda_0.
\]
Thus, $\Lambda_*\leq\Lambda_0$.
If $ 2\sqrt{c_\gamma} (1-2\gamma)=0$, we know that $\Lambda_0=0$
and $\Lambda_*\leq \frac{-\vartheta}
{\Qcal_\perp(\zeta)}\leq0$.  The proof is completed.
\end{proof}

Lemmas~\ref{lem:bounded-max-seq} and
\ref{thm:concentration-bound} show that $\Lambda_*$ is attained
if $\Lambda_*>\Lambda_0$. Therefore, it remains to establish
this strict inequality.
\begin{lemma}
\label{prop:semi-trivial-endpoint}
Define
\begin{equation*}
\Lambda_{1}
:=
\frac{a^*\max\{s_1^2,s_2^2\}}
{4(s_1+s_2)^4\vartheta},
\end{equation*}
then $\Lambda_*\geq\Lambda_{1}>\Lambda_0$.
\end{lemma}

\begin{proof}
First, we give the lower bound for $\Lambda_*$ from semi-trivial states.
Recall that
\[
\Lambda_*
=
\sup_{\substack{U\in\mathcal M\\
\mathcal P(U)<c_\gamma}}
\frac{c_\gamma-\mathcal P(U)}
{\mathcal F(U)}.
\]
For $\varepsilon>0$, set
\[
\phi_{\varepsilon}(x)
:=
\frac{1}{\varepsilon}
q\left(\frac{x}{\varepsilon}\right).
\]
Since $\|q\|_2=1$, we have $\|\phi_{\varepsilon}\|_2=1.$
Moreover,
\[
\mathcal F(\phi_{\varepsilon},0)
=
\frac{s_1^2}{a^*\varepsilon^2}\ \text{and}\ \mathcal P(\phi_{\varepsilon},0)
=
\vartheta\varepsilon^2.
\]
Therefore,
\[
\Lambda_*
\geq\frac{c_\gamma-\mathcal P(\phi_{\varepsilon},0)}
{\mathcal F(\phi_{\varepsilon},0)}
=
\frac{a^*\varepsilon^2}{s_1^2}
\left(
c_\gamma-\vartheta\varepsilon^2
\right).
\]
Choosing $\varepsilon^2=\frac{c_\gamma}{2\vartheta}$, we get that
\[
\Lambda_*
\geq
\frac{a^*c_\gamma^2}
{4s_1^2\vartheta}.
\]
Note that
$c_\gamma=\gamma(1-\gamma)$ and $1-\gamma=\frac{s_1}{s_1+s_2}$,
then
\begin{equation}
\label{eq:semi-first-component}
\Lambda_*
\geq
\frac{a^*s_2^2}
{4(s_1+s_2)^4\vartheta}.
\end{equation}
Applying the same argument to $(0,\phi_{\varepsilon,e_1})$ with
 \[
\phi_{\varepsilon,e_1}(x):
=
\frac{1}{\varepsilon}
q\left(\frac{x-e_1}{\varepsilon}\right),
\]
 we obtain that
 \begin{equation}
\label{eq:semi-second-component}
\Lambda_*
\geq
\frac{a^*s_1^2}
{4(s_1+s_2)^4\vartheta}.
\end{equation}
Combining \eqref{eq:semi-first-component} and \eqref{eq:semi-second-component}, we conclude that
$$
\Lambda_*\geq\Lambda_{1}
=
\frac{a^*\max\{s_1^2,s_2^2\}}
{4(s_1+s_2)^4\vartheta}.
$$

Next, we prove that $\Lambda_1>\Lambda_0$ for any $a_1,a_2\in (0,a^*)$. Recall that
\begin{equation*}
\Lambda_0
=
\frac{c_\gamma(1-2\gamma)^2}
{\vartheta}
J_\kappa,
\end{equation*}
where
$\kappa=1-\frac{2s_1s_2}{a^*}\in(-1,1).$
Since
\[
c_\gamma
=
\frac{s_1s_2}{(s_1+s_2)^2}
\ \text{and}\
1-2\gamma
=
\frac{s_1-s_2}{s_1+s_2},
\]
then
\[
c_\gamma(1-2\gamma)^2
=
\frac{s_1s_2(s_1-s_2)^2}
{(s_1+s_2)^4}.
\]
Thus,
\[
\Lambda_0
=
\frac{s_1s_2(s_1-s_2)^2}
{(s_1+s_2)^4\vartheta}
J_\kappa
\]
and $\Lambda_{1}>\Lambda_0$
is equivalent to
\begin{equation}
\label{eq:semi-comparison-basic}
\frac{a^*}{4}\max\{s_1^2,s_2^2\}
>
s_1s_2(s_1-s_2)^2J_\kappa.
\end{equation}

Now, we distinguish the two possible signs of $\kappa$.

\medskip
\noindent
\textbf{Case 1: $\kappa\leq0$.} For $\kappa\leq0$,
recall from \eqref{eq:J-variational} that
\begin{equation*}
J_\kappa
=
\sup_{\zeta\in H^1(\R^2)\setminus\{0\}}
\frac{\left(\int_{\R^2}q\zeta \dd x\right)^2}
{\Qcal_\perp(\zeta)},
\end{equation*}
where
$$
\Qcal_\perp(\zeta)=\int_{\R^2}
\left(
|\nabla\zeta|^2+\zeta^2-\kappa Q^2\zeta^2
\right)\dd x\geq\int_{\R^2}
\left(|\nabla\zeta|^2+\zeta^2\right)\dd x.
$$
By the H\"older's inequality, we have
$J_\kappa\leq 1.$
Note that
$s_1s_2<a^*$, then
 \eqref{eq:semi-comparison-basic} holds if
\[
\frac14\max\{s_1^2,s_2^2\}
>
(s_1-s_2)^2.
\]
This is equivalent to
$$
\frac{\min\{s_1,s_2\}}{\max\{s_1,s_2\}}>1/2.
$$
On the other hand, $\kappa\leq0$ means $2s_1s_2\geq a^*>(\max\{s_1,s_2\})^2$,
then the above inequality holds.
Therefore, if $\kappa\leq0$, then \eqref{eq:semi-comparison-basic} holds and $\Lambda_{1}>\Lambda_0.$

\medskip
\noindent
\textbf{Case 2: $0<\kappa<1$.}
For $0<\kappa<1$,
recall that $
L_\perp\psi_\kappa=q
\ \text{in }H^{-1}(\R^2)
$
and
$$
J_\kappa
=
\langle q,\psi_\kappa\rangle
=
\Qcal_\perp(\psi_\kappa)
>0.
$$
Then
\begin{equation}\label{eq:Jkappa-energy}
J_\kappa
=
\|\nabla\psi_\kappa\|_2^2
+\|\psi_\kappa\|_2^2
-\kappa\int_{\mathbb R^2}Q^2\psi_\kappa^2.
\end{equation}
By H\"older's inequality and the
Gagliardo--Nirenberg inequality \eqref{eqn:GN},
we have
\begin{align*}
\int_{\mathbb R^2}Q^2\psi_\kappa^2
&\le
\|Q\|_4^2\|\psi_\kappa\|_4^2\\
&\le
(2a^*)^{1/2}
\left(\frac{2}{a^*}\right)^{1/2}
\|\psi_\kappa\|_2\|\nabla\psi_\kappa\|_2\\
&=
2\|\psi_\kappa\|_2\|\nabla\psi_\kappa\|_2,
\end{align*}
which, together with \eqref{eq:Jkappa-energy}, implies that
\begin{align*}
J_\kappa
&\ge
\|\nabla\psi_\kappa\|_2^2
+\|\psi_\kappa\|_2^2
-2\kappa
\|\psi_\kappa\|_2\|\nabla\psi_\kappa\|_2\\
&=
\bigl(
\|\nabla\psi_\kappa\|_2
-\kappa\|\psi_\kappa\|_2
\bigr)^2
+
(1-\kappa^2)\|\psi_\kappa\|_2^2\\
&\ge
(1-\kappa^2)\|\psi_\kappa\|_2^2.
\end{align*}
On the other hand, since $\|q\|_2=1$, the Cauchy--Schwarz inequality gives
\[
J_\kappa
=
\langle q,\psi_\kappa\rangle
\le
\|\psi_\kappa\|_2.
\]
Therefore,
\[
J_\kappa
\ge
(1-\kappa^2)\|\psi_\kappa\|_2^2
\ge
(1-\kappa^2)J_\kappa^2,
\]
which means that
\[
J_\kappa\le\frac{1}{1-\kappa^2}.
\]
Note that
\[
1-\kappa=\frac{2s_1s_2}{a^*} \ \text{and}\
1+\kappa=\frac{2(a^*-s_1s_2)}{a^*},
\]
then
\[
J_\kappa
\le
\frac{(a^*)^2}{4s_1s_2(a^*-s_1s_2)}.
\]
Now, we prove \eqref{eq:semi-comparison-basic} holds. This is equivalent to verifying that
\[
\frac{a^*}{4}\max\{s_1^2,s_2^2\}
>
\frac{(s_1-s_2)^2(a^*)^2}{4(a^*-s_1s_2)},
\]
that is,
\[
a^*-\max\{s_1,s_2\}\min\{s_1,s_2\}
>
a^*\left(1-\frac{\min\{s_1,s_2\}}{\max\{s_1,s_2\}}\right)^2.
\]
After expanding, this becomes
\begin{equation}\label{eqn:expanding-case-2}
(\max\{s_1,s_2\})^2
<
a^*\left(2-\frac{\min\{s_1,s_2\}}{\max\{s_1,s_2\}}\right).
\end{equation}
Note that
\[
2-\frac{\min\{s_1,s_2\}}{\max\{s_1,s_2\}}\ge 1,
\]
and
\[
(\max\{s_1,s_2\})^2=\max\{s_1^2,s_2^2\}<a^*.
\]
Hence, \eqref{eqn:expanding-case-2} holds and then \eqref{eq:semi-comparison-basic} holds for \(0<\kappa<1\).
Consequently, for any $a_1,a_2\in (0,a^*)$, we obtain  $\Lambda_*\geq\Lambda_{1}>\Lambda_0.$
\end{proof}

\noindent\textbf{Proof of Theorem \ref{thm:main-2}.}
This  directly follows from
Lemmas \ref{lem:bounded-max-seq}-\ref{prop:semi-trivial-endpoint}.
\qed

\ \\
\noindent\textbf{\Large Acknowledgements}

\noindent {This research is supported by National Natural Science Foundation of China (No.12371120).}
\vspace{0.8em}

\noindent\textbf{\Large Declaration}

\noindent {{\bf Conflict of interest} The authors do not have conflict of interest.}

\end{document}